\documentclass[10pt]{amsart}

\usepackage[T1]{fontenc}
\usepackage{amsmath,amssymb,amsthm}
\usepackage{graphicx}
\usepackage{float}
\usepackage{booktabs}
\usepackage[dvipsnames]{xcolor}
\usepackage[colorlinks=true,linkcolor=blue,citecolor=blue,urlcolor=blue]{hyperref}

\usepackage[left=2.1cm,right=2.1cm,top=2.9cm,bottom=2.9cm]{geometry}

\newcommand{\rev}[1]{\textcolor{black}{#1}}
\newenvironment{revblock}{\color{black}}{}

\theoremstyle{definition}

\theoremstyle{remark}
\newtheorem{remark}{Remark}
\theoremstyle{plain}
\newtheorem{theorem}{Theorem}
\newtheorem{proposition}[theorem]{Proposition}
\newtheorem{corollary}[theorem]{Corollary}
\newtheorem{lemma}[theorem]{Lemma}

\newcommand{\R}{\mathbb{R}}
\newcommand{\N}{\mathbb{N}}
\DeclareMathOperator{\dist}{dist}

\begin{document}

\title{A PINNs approach for the computation of eigenvalues in elliptic problems}

\author{Juli\'an Fern\'andez Bonder}
\address{Departamento de Matem\'atica, Facultad de Ciencias Exactas y Naturales, Universidad de Buenos Aires, Buenos Aires, Argentina, and Instituto de C\'alculo UBA-CONICET, Buenos Aires, Argentina}
\email{jfbonder@dm.uba.ar}

\author{Ariel Salort}
\address{Departamento de Matem\'aticas Aplicadas y Ciencia de Datos, Universidad San Pablo-CEU, CEU Universities, Juli\'an Romea 2, 28003 Madrid, Spain}
\email{ariel.salort@ceu.es}

\begin{abstract}
\begin{revblock}
We propose a physics-informed neural network (PINN) method for eigenvalue problems of elliptic
operators that relies neither on the Rayleigh quotient nor on orthogonality. For each value $E$ of the
spectral parameter, the squared residual of the equation is minimized over network functions that
satisfy the Dirichlet condition exactly, and the eigenvalues are detected as the minima of the
resulting loss curve. For self-adjoint operators the loss curve is the squared distance from $E$ to
the spectrum. For the penalized loss used in practice we compute the loss curve in closed form, which
excludes spurious local minima and yields an explicit rule for the penalty parameter; we prove that
the detection is stable under training and sampling errors, provided the measured training error is
small compared with the squared gaps, and in the linear case every computed eigenvalue carries an
a posteriori Weinstein and Kato--Temple enclosure. For the $p$-Laplacian the zeros of the loss curve
are exactly the eigenvalues and the curve lies below a parabola at each of them. Experiments
agree with exact values in dimensions two to four and with independent finite element (dimensions
two and three) and shooting computations, including $p$-Laplacian eigenvalues beyond the second; in dimension ten, where
mesh-based methods are not practical, the method locates the first two eigenvalues to within $1\%$
(single run, weak certificate). The method does not resolve
multiplicities and, for the lowest eigenvalues of linear problems, it is more expensive than
Rayleigh-quotient methods with deflation.


\end{revblock}
\end{abstract}

\maketitle

\section{Introduction}

The impact of artificial intelligence (AI) on everyday life has become undeniable in recent years. From smartphones to image recognition, from medical diagnostics to self-driving cars, machine learning, and deep learning in particular, has transformed the way we interact with technology.

This influence has also extended into mathematics and physics, where machine learning techniques are being applied to analyze complex phenomena that were traditionally approached using classical analytical methods. Among these, problems governed by partial differential equations (PDEs) are especially prominent, due to their central role in modeling physical systems across many disciplines. Reflecting this trend, a growing number of software libraries and frameworks have been developed as educational and research tools for solving PDE-based problems in computational science and engineering; see, for instance, \cite{Lu-Meng-Mao, Rackauckas, Rackauckas-Innes-Ma-Bettencourt-etal} and references therein.

In parallel, neural network-based approaches for approximating solutions to PDEs have gained
increasing attention. These include successful applications to problems such as the Burgers,
Eikonal, and heat equations \cite{Blechschmidt-Ernst}, linear diffusion equations in complex
two-dimensional geometries \cite{Berg-Nystrom}, and nonlinear problems such as three-dimensional
PDEs, Emden--Fowler systems, and fractional-order predator--prey models \cite{GT24, KKG, MHO}.
\rev{Nonlinear dispersive models, such as the generalized higher-order nonlinear Schr\"odinger
equation, whose multi-peak and breather-type solitary waves were recently studied in
\cite{Yasin-etal}, are further natural targets for such solvers.}

A natural extension of this idea is the use of neural networks (NNs) as general-purpose function approximators for solving differential equations. A prototypical example arises in quantum mechanics, where one seeks to solve the eigenvalue problem
\begin{equation}\label{eigen}
-\Delta u + V(x) u = E u,\quad \text{in } \R^{\rev{d}},
\end{equation}
with a given potential function \( V \colon \R^{\rev{d}} \to \R \) and eigenvalue parameter \( E \in \R \). The idea of employing NNs to approximate solutions of such equations dates back to the 1990s \cite{Lagaris-Likas-Fotiadis}, but the field gained significant momentum with the introduction of {\em Physics-Informed Neural Networks} (PINNs) in \cite{Raissi-Perdiakaris-Karniadakis}, which triggered a wave of research and applications (see also \cite{Karniadakis-etal}).
\rev{In all the examples below the problem is posed in a bounded domain $\Omega$ with homogeneous
Dirichlet boundary conditions, also when a potential $V$ is present.}

Most early PINN-based work focused on source problems, where the term \( E u \) in \eqref{eigen} is replaced by a known forcing function \( f \). The eigenvalue problem itself has been more recently addressed: first in one dimension \cite{Jin-Mattheakis-Protopapas1, Jin-Mattheakis-Protopapas2, Sakar}, and later in two dimensions \cite{Holliday-Lindner-Ditto}. Related approaches have also been proposed in \cite{Rowan-Evans-Maute-Dootsan}, which include extensions to nonlinear eigenvalue problems\rev{; the Deep Ritz method \cite{E-Yu} also treats eigenvalue problems through the Rayleigh quotient}.

Several methodologies have been proposed to address eigenvalue problems using neural networks.
The approach in \cite{Sakar} constructs a trial wavefunction that automatically satisfies the boundary conditions, represents the unknown component with a neural network, and trains the network by minimizing an appropriately defined loss function.
In \cite{Jin-Mattheakis-Protopapas1, Jin-Mattheakis-Protopapas2}, the authors introduce an unsupervised neural network architecture that simultaneously learns both the eigenvalues and the corresponding eigenfunctions through a scanning mechanism.
Alternatively, \cite{Rowan-Evans-Maute-Dootsan} employs a neural network to discretize the eigenfunction and minimize the associated Rayleigh quotient.

\begin{figure}[htbp]
\centering
\includegraphics[width=.4\columnwidth]{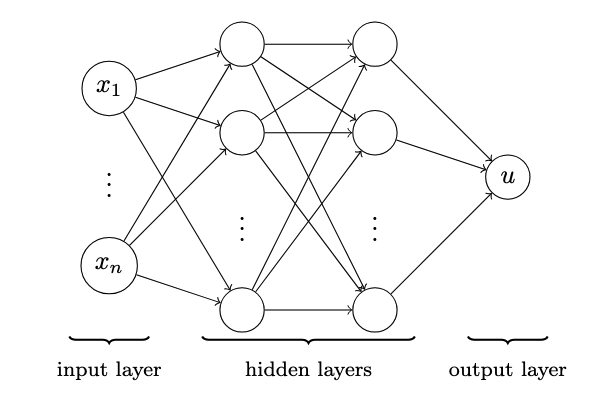}
\caption{A fully connected feedforward NN with a \rev{$d$}-dimensional input and a 1-dimensional output.}
\label{NN.fig}
\end{figure}

In these works, the solution \( u \) is approximated using a fully connected feedforward neural network (see Fig.~\ref{NN.fig}), with variations in the number of hidden layers and activation functions across implementations. To approximate the eigenvalue \( E \), the authors minimize the {\em Rayleigh quotient}:
\[
\rev{R}(u) = \frac{\int_{\rev{\Omega}} |\nabla u|^2 + V(x) u^2\, dx}{\int_{\rev{\Omega}} u^2\, dx}.
\]
Minimizing \( \rev{R}(u) \) yields an approximation of the principal eigenvalue \( E_1 \) and its associated eigenfunction \( u_1 \). Higher-order eigenpairs can be computed by iteratively minimizing the same functional while constraining the search to subspaces orthogonal to previously computed eigenfunctions.

This approach, while effective, presents two main limitations:
\begin{enumerate}
\item It is inherently sequential: to compute eigenvalues within a given range, one must first compute all lower ones, which can be inefficient.
\item Its extension to nonlinear problems (such as those involving the \(p\)-Laplace operator) is limited, since orthogonality conditions between eigenfunctions are no longer available.
\end{enumerate}

\begin{revblock}
\noindent\emph{Contributions.} (i) We propose to compute eigenvalues as the zeros of the loss curve
$E\mapsto\inf_u\Lambda(u,E)$, where $\Lambda$ is the squared PDE residual; for self-adjoint
problems $\Lambda(E)=\dist(E,\sigma)^2$ (Section~\ref{sec.approach}). (ii) We compute the loss curve
of the penalized loss used in practice in closed form, show that it has no spurious local minima,
derive a rule for the penalty parameter $\mu$ and test it on windows with one and two eigenvalues
(Proposition~\ref{prop.penalized}, Section~\ref{sec.mu}). (iii) We prove a stability result for the
detection of eigenvalues under training and sampling errors, measure the error level it assumes,
and attach to each computed eigenvalue an a posteriori enclosure (Weinstein and Kato--Temple;
Sections~\ref{sec.certificate} and \ref{sec.stability}). (iv) We validate the method against exact
values in dimensions $2$--$4$ and independent finite element computations in dimensions $2$ and
$3$, with several seeds
(Section~\ref{sec.examples}), quantify its cost and compare it with a Rayleigh-quotient PINN with
deflation (Section~\ref{sec.cost}). (v) For the $p$-Laplacian we prove that the zeros of the loss
curve are exactly the eigenvalues and that the curve satisfies the upper bound
$\Lambda_p(E)\le\inf_kc_k(E-E_k)^2$, so that every isolated eigenvalue lies at the bottom of a well; no lower
bound (and hence no a posteriori certificate) is available. We compute eigenvalues beyond the second
one, compare them with independent references and analyse the sources of error
(Section~\ref{sec.nolin}).

\noindent\emph{Organization of the paper.} Section~\ref{sec.approach} introduces the loss curve and
the penalized loss. Section~\ref{sec.implementation} describes the implementation, the refinement
of the computed eigenvalue and the a posteriori certificate. Section~\ref{sec.examples} contains
the linear examples (higher eigenvalues, higher dimensions, other geometries and potentials, with
finite element validation). Section~\ref{sec.nolin} treats the $p$-Laplacian. Section~\ref{sec.analysis}
analyses the practical method: stability and spurious minima, choice of $\mu$, and computational
cost. Section~\ref{sec.conclusions} summarizes the results and the limitations. The appendices
contain the proof of existence of minimizers (Theorem~\ref{teo.min}) and the proofs of
Propositions~\ref{prop.penalized}, \ref{prop.rayleigh}, \ref{prop.plap}, \ref{prop.ck} and \ref{prop.stability}.
\end{revblock}

\section{Our approach}\label{sec.approach}

To overcome the drawbacks mentioned above, we propose a different approach to problem \eqref{eigen}. We define the following loss function:
\begin{equation}\label{Lambda}
\Lambda(u, E) = \frac{\int_{\rev{\Omega}} (\Delta u - V(x) u + Eu)^2\, dx}{\int_{\rev{\Omega}} u^2\, dx}.
\end{equation}
Given a fixed $E\in \R$, minimizing with respect to $u$ yields a {\em loss curve}:
\begin{equation}\label{loss-curve}
\Lambda(E) = \inf_u \Lambda(u, E).
\end{equation}
A proof of the existence of a minimizer $u_E$ for \eqref{loss-curve} is provided in Theorem \ref{teo.min} in the appendix.

\begin{revblock}
Let $L=-\Delta+V$, with the Dirichlet condition built into its domain $\mathrm{Dom}(L)$, be self-adjoint with
compact resolvent (this is the case for a bounded domain and bounded $V$), let
$E_1\le E_2\le\cdots$ be its eigenvalues repeated according to multiplicity, $\{u_k\}$ an orthonormal
basis of eigenfunctions, $\sigma=\{E_k\}$ the spectrum and
\[
\ell(E):=\dist(E,\sigma)
\]
the distance from $E$ to the spectrum. Writing
$u=\sum_ka_ku_k\in \mathrm{Dom}(L)$,
\[
\Lambda(u,E)=\frac{\sum_k a_k^2 (E_k-E)^2}{\sum_k a_k^2}\ \ge\ \min_k (E_k-E)^2,
\]
with equality if and only if $u$ lies in the eigenspace of the eigenvalue(s) closest to $E$. Hence
\begin{equation} \label{cota.igual}
\Lambda(E)=\min_k (E_k-E)^2=\ell(E)^2 .
\end{equation}
In particular $\Lambda\ge0$, the zeros of $\Lambda$ are exactly the eigenvalues, every local minimum
of $\Lambda$ is a zero, and $\Lambda$ is a parabola of unit curvature around each eigenvalue. The
curve does not see multiplicities: a multiple eigenvalue produces a single zero.
\end{revblock}

Note that the characterization of eigenvalues as the zeros of $\Lambda$ only uses that
$\Lambda(u,E)$ is the squared residual of the equation, and not the spectral decomposition
above. Hence the same mechanism applies to nonlinear eigenvalue problems, where no
orthogonality between eigenfunctions is available, as shown in \rev{Section~\ref{sec.nolin}}. In the nonlinear case there is no spectral decomposition, so the identity
\eqref{cota.igual} is no longer available. What survives is the mechanism
itself: the loss vanishes exactly at the eigenvalues \rev{(Proposition~\ref{prop.plap})}, so they can still be
detected as the minima of the loss curve\rev{, and near each eigenvalue the curve lies below a parabola,
$\Lambda_p(E)\le c_k(E-E_k)^2$}.

Having defined $\Lambda(E)$, the eigenvalues can be found just by scanning $E$ and locating the local minima of $\Lambda$\rev{.
The loss at a given $E$ does not involve any other eigenpair: no orthogonalization against previously
computed eigenfunctions is needed, and errors in one eigenvalue do not propagate to the next. This
is the main strength of the proposed method. (In practice the scan is warm-started from one node to
the next; training from scratch at a high value of $E$ is unreliable, see Section~\ref{sec.mu}.)}
\rev{The idea of scanning a spectral parameter and locating the minima of a residual is classical
in other contexts (e.g.\ the method of particular solutions \cite{Betcke-Trefethen} and the folded-spectrum method \cite{Wang-Zunger}); here
it is combined with a neural network ansatz, which makes it applicable in several dimensions and to
nonlinear operators.}

In Fig.~\ref{Lambda.fig}, the \rev{curve $\Lambda(E)$} for problem \eqref{eigen} is plotted in the two-dimensional case with \rev{the hard-wall} potential given by $V(x)=0$ for $|x|\le 1$ and $V(x)=\infty$ for $|x|>1$.  In this case the eigenvalues are known explicitly. The first four \rev{distinct eigenvalues} are:
$$
 E_1 = 5.7832,\ E_2 = \rev{14.6820},\ E_3 =  26.3746,\ E_4 =  30.4713
 $$
\rev{($E_2$ and $E_3$ are double).}
\begin{figure}[htbp]
\centering
\includegraphics[width=.45\columnwidth]{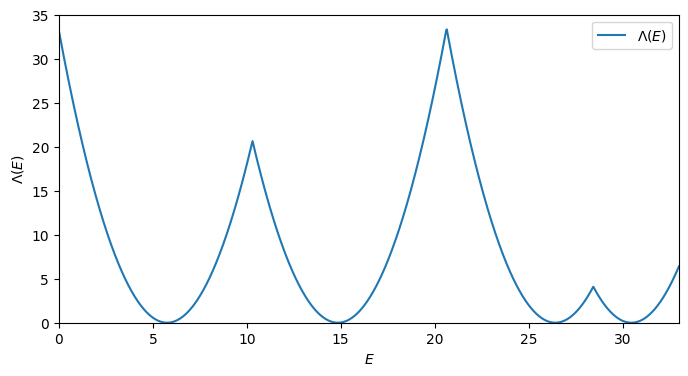}
\caption{The \rev{exact} loss curve $\Lambda(E)\rev{=\min_k(E-E_k)^2}$ \rev{for the Dirichlet Laplacian in the unit disk}.}
\label{Lambda.fig}
\end{figure}
Our approach proceeds as follows:

We first fix the interval $[E_*, E^*]$  where the eigenvalues are expected to lie. Then, we take a partition of this interval\rev{ (it need not be uniform; for simplicity we use a uniform step $h$ in the experiments and in Proposition~\ref{prop.stability})}:
$$
E_*=E^1<E^2<\cdots<E^{\rev{J}}=E^*,
$$
and for each $E^i$ we train our neural network to minimize $\Lambda(u, E^i)$, obtaining a corresponding minimizer $u^i$.

\rev{Finally, we look for the local minima of the discrete loss curve $\Lambda(E^i)$ whose value is below a threshold. By \eqref{cota.igual} and Proposition~\ref{prop.stability} below, each of them lies within a distance of an eigenvalue controlled by the training error $\delta$, which can be estimated a posteriori; the eigenvalue is then refined from the trained network (Section~\ref{sec.certificate}), and the associated minimizers $u^i$ approximate the corresponding eigenfunctions.}

\rev{The aim of this article is to give an approach to compute eigenvalues based on neural networks
whose implementation is simple and which does not require orthogonalization against other eigenfunctions.
Our goal is not to obtain the best estimates for linear problems (for that aim there are efficient
and well-studied methods based on finite differences/elements, see Section~\ref{sec.fem}). However,
the method is flexible and has the following features:}
\begin{enumerate}
\item  Unlike the aforementioned methods, \rev{our approach does not orthogonalize against lower eigenfunctions, so the accuracy of a higher eigenvalue does not depend on that of the lower ones (the scan still has to reach the window, normally by warm-start continuation from below).}

\item \rev{The implementation is the same in every dimension; we report experiments in dimensions two, three and four.}

\item \rev{The same code handles different potentials and geometries; only the boundary factor $B$ changes.}

\item The neural network architecture employed in our implementation is simple and extends naturally to nonlinear cases. \rev{While we do not claim more accuracy than existing methods, this allows us to compute eigenvalues of the $p$-Laplacian beyond the second one (Section~\ref{sec.nolin}).}

\item \rev{Each computed eigenvalue comes with an a posteriori bound in the linear case (Section~\ref{sec.certificate}), and the computations at different values of $E$ are independent and can be run in parallel (Section~\ref{sec.cost}).}

\item \rev{In principle, the approach can be adapted to more general operators beyond the $p$-Laplacian.}
\end{enumerate}
\rev{Its limitations (cost compared with Rayleigh-quotient methods for the lowest eigenvalues,
multiplicities and close clusters not resolved) are discussed in Sections~\ref{sec.analysis} and
\ref{sec.conclusions}.}

\section{Implementation details}\label{sec.implementation}

We consider a fully connected neural network, implemented as a feed-forward multilayer perceptron with three hidden layers. Each hidden layer contains 128 neurons with $\tanh$ activation functions. The network maps a spatial coordinate $x\in \R^d$ to a scalar output and is defined as $
x \mapsto \text{NN}(x) \in \mathbb{R}$. All linear layers are initialized using Xavier uniform initialization with the appropriate gain for the tanh activation, and all biases are set to zero.

To enforce the homogeneous Dirichlet boundary condition on the domain $\Omega$, the raw network output $\text{NN}(x)$ is multiplied by a boundary factor $B(x)$, chosen so that
$$
B(x)=0 \text{ for all } x\in\partial\Omega, \qquad B(x)>0 \text{ for all } x\in\Omega.
$$
This construction ensures that the final approximation automatically vanishes on the boundary. Thus, the network returns
$$
u(x)=B(x)\,\text{NN}(x).
$$
\begin{revblock}
The factors used for the domains of this paper are smooth functions vanishing exactly on
$\partial\Omega$ and comparable to the distance to the boundary:
\[
\begin{array}{ll}
\text{Domain} & B(x)\\ \hline
\text{Ball } \{|x|<1\}\subset\R^d,\ d=2,3,4 & 1-|x|^2\\
\text{Square } (0,a)^2,\ a=1,2 & x_1(a-x_1)\,x_2(a-x_2)\\
\text{Triangle with vertices } (0,0),(1,0),(0,1) & x_1x_2(1-x_1-x_2)\\
\text{Annulus } \{\tfrac12<|x|<1\} & (|x|-\tfrac12)(1-|x|)
\end{array}
\]
Since $u=B\,\mathrm{NN}$ satisfies the boundary condition exactly and is smooth, every network
function belongs to the domain $\mathrm{Dom}(L)$ of the operator; this is what makes the a posteriori bounds of
Section~\ref{sec.certificate} applicable to the trained networks.
\end{revblock}

We begin by studying the \rev{simplest} case, \rev{the hard-wall potential}
\begin{equation}\label{conf.V}
V(x) = \begin{cases}
0 & \text{if } x\in \Omega\\
\infty & \text{if } x\not\in \Omega,
\end{cases}
\end{equation}
which is equivalent to solving the Helmholtz problem with homogeneous Dirichlet boundary condition:
\begin{equation}\label{eigen2}
\begin{cases}
-\Delta u = E u & \text{in }\Omega\\
u=0 & \text{on }\partial\Omega.
\end{cases}
\end{equation}

To ensure normalization, we include a penalization term in the loss function to enforce $\|u\|_2=1$. The resulting loss is:
\begin{equation}\label{lossmu}
\bar \Lambda(u, E) = \int_\Omega (\Delta u \rev{-Vu} + Eu)^2\, dx + \mu\left(\int_\Omega u^2\, dx - 1\right)^2.
\end{equation}
All integrals are approximated using a Monte Carlo scheme, typically using 5000 sampling points.
\rev{(In the code the integrals are replaced by sample means, i.e.\ by integrals with respect to
$dx/|\Omega|$; all statements below are invariant under this change of measure.) For the unit disk
we used $\mu=100$; the choice of $\mu$ is analysed in Section~\ref{sec.mu}.}

\begin{revblock}
\begin{proposition}[Loss curve of the penalized loss]\label{prop.penalized}
Let $\mu>0$, $E\in\R$ and $\ell=\ell(E)$. Then
\[
\bar\Lambda(E):=\inf_{u\in \mathrm{Dom}(L)}\bar\Lambda(u,E)=
\begin{cases}
\ell^2-\dfrac{\ell^4}{4\mu} & \text{if } \ell^2\le 2\mu,\\[2mm]
\mu & \text{if } \ell^2\ge 2\mu .
\end{cases}
\]
If $\ell^2<2\mu$ the minimizers are the functions $u$ in the eigenspace(s) of the eigenvalue(s)
closest to $E$ with $\|u\|_2^2=1-\ell^2/(2\mu)$; for them $\Lambda(u,E)=\ell^2$ and the residual term
equals $\ell^2(1-\ell^2/(2\mu))$. If $\ell^2\ge2\mu$ the unique minimizer is $u=0$. Moreover, for every $E$
all local minimizers of $\bar\Lambda(\cdot,E)$ in $\mathrm{Dom}(L)$ are global minimizers, and every other
critical point is a saddle point.
\end{proposition}
The proof is given in Appendix~\ref{app.proofs}. Thus, at the level of the continuous problem the
landscape has no spurious local minima, the zeros of $\bar\Lambda$ are exactly the eigenvalues for
every $\mu>0$, and the only artifact introduced by the penalty is the trivial state $u=0$, which is
the minimizer when $\ell(E)^2\ge2\mu$. Spurious minima observed in practice can therefore only come
from the network parametrization, the optimizer or the sampling (Section~\ref{sec.stability}).
\end{revblock}

As described in the previous section, we train the NN to minimize $\bar\Lambda(u, E)$ for each value in a discrete partition $E_* = E^1<\cdots<E^{\rev{J}}=E^*$. To improve efficiency, for each $E^i$, the network is initialized with the weights and biases obtained from training at $E^{i-1}$ \rev{(warm start); in some scans the network is re-initialized every 15--60 grid points}.

\rev{The parameters are optimized with Adam, constant learning rate $10^{-3}$, and at most $N_{\rm ep}$
Adam steps (epochs) per node; at each epoch a fresh sample of $N$ points ($N=5000$ unless stated) is
drawn to compute the loss. Training at a node stops early if the loss does not decrease by
$\mathrm{tol}=10^{-5}$ in 300 consecutive epochs. The window, the grid step, $\mu$, $N_{\rm ep}$, the
type of start and the number of seeds of every example are listed in Table~\ref{table.settings},
together with the regime of $\mu$ with respect to the guideline of Section~\ref{sec.mu}.} Upon completion of training, we obtain the discrete loss curve
$$
\bar \Lambda(E^i)=\bar\Lambda(u^i, E^i)\rev{,}
$$
and each local minimum below a given threshold $\epsilon>0$ is taken as an approximate eigenvalue for \eqref{eigen2}.
\rev{The threshold was chosen per example after inspecting the curve; Proposition~\ref{prop.stability}
shows what any choice guarantees, and the certificate of Section~\ref{sec.certificate} makes the
threshold unnecessary for accepting a detection. In the tables we report, besides the computed
eigenvalue, the value $\mathrm{Residual}=\frac1N\sum_{i=1}^N(\Delta u-Vu+Eu)^2(x_i)$ of the first
term of \eqref{lossmu}; since $\frac1N\sum u(x_i)^2\approx1$ at the minima,
$\mathrm{Residual}\approx\Lambda(u,E)=r^2$, where $r$ is the residual norm defined in
Section~\ref{sec.certificate}.}

\begin{table}[tbp]
\centering\scriptsize
 
\caption{Settings of the linear examples ($3\times128$ network, Adam with learning rate $10^{-3}$,
early stopping 300/$10^{-5}$). $h$: grid step; $N_{\rm ep}$: maximal number of epochs per node;
start: warm (continuation from the previous node), cold (fresh network), re-init.\ (fresh network
every $n$ nodes); $D=\max\ell(E)$ over the window (Section~\ref{sec.mu}); regime: (ii) faithful
($\mu\ge D^2$), (iii) $\frac12D^2\le\mu<D^2$, (iv) collapse possible ($\mu<\frac12D^2$), see
Corollary~\ref{cor.mu}. Single runs unless the number of seeds is given (seed 0 where the seed was
recorded). $^a$Window and step of the accompanying notebook (approximate for the curves of the
earlier figures). $^b$Reduced network $3\times64$, $N=3000$.}
\label{table.settings}
\begin{tabular}{@{}llccccll@{}}
\toprule
example & window & $h$ & $\mu$ & $N_{\rm ep}$ & start, seeds & $D^2$ & regime\\
\midrule
disk (Fig.~\ref{lossNN.fig}) & $[0,32]^a$ & -- & 100 & 3000 & warm & 33.4 & (ii)\\
square (Fig.~\ref{eigen55.fig}) & $[44,55]$ & -- & 3000 & -- & cold/warm & 31.9 & (ii)\\
square (Tab.~\ref{table.1}) & $[0,105]$ & 1.33 & 500 & 5000 & re-init.\ 25 & 390 & (ii)$^\ast$\\
square (\S\ref{sec.stability}, \S\ref{sec.cost}) & $[15,55]$, $[17.3,102.3]$ & 5 & 450 & 3000 & warm, 3 seeds & 219 & (ii)\\
square ($\mu$ study) & $[40,60]$, $[15.6,53.1]$ & 2, 2.5 & 1--$10^4$ & $2000^b$ & warm, 2 seeds & 114, 219 & all\\
square (cold start) & $[95,102]$ & 1.75 & 450--$2\cdot10^4$ & 5000 & cold, 1--3 seeds & -- & --\\
ball $\R^3$ (Fig.~\ref{loss3D.fig}) & $[0,70]^a$ & -- & 100 & 2000 & re-init.\ 15 & 97.4 & (ii)\\
ball $\R^4$ (Fig.~\ref{loss3D.fig}) & $[0,55]^a$ & -- & 200 & 2000 & warm & 216 & (iii)\\
triangle (Fig.~\ref{loss.lap.triang.fig}), $N\!=\!4000$ & $[0,180]$ & 1.8$^a$ & 15000 & 2000 & re-init.\ 50$^a$ & 2435 & (ii)\\
Tabs.~\ref{table.2}, \ref{table.4} & 3 nodes/eigenv. & 1 & as above & 2000 & cold/warm & $\le1.2$ & (ii)\\
$\omega=1$ (Fig.~\ref{pot2.fig}) & $[0,55]^a$ & 0.7 & 100 & 3000 & re-init.\ 50 & 34.7 & (ii)\\
$\omega=10$ (Fig.~\ref{pot1.fig}) & $[0,55]^a$ & 0.7 & 100 & 3000 & re-init.\ 50 & 204 & (iv)\\
wells (Fig.~\ref{fig.well-two}) & $[0,75]^a$ & 0.63 & 400/1000 & 2000 & re-init.\ 50 & 92/97 & (ii)\\
FEM windows (Tab.~\ref{table.fem}) & $4\times7$ nodes & 0.5 & 400/1000/100 & 2000/3000 & cold/warm, 2--3 seeds & $\le 2.3$ & (ii)\\
blind scan (Tab.~\ref{table.blind}) & $[4.7,49.7]$ & 1.5 & 400 & 2000 & warm & 41.1 & (ii)\\
ball $\R^{10}$ (\S4.2) & $[54,84]$ & 3 & 2000 & 2000 & warm & 93.7 & (ii)\\
\bottomrule
\multicolumn{8}{@{}l}{$^\ast$The re-initializations at $E\approx33,66,100$ are cold starts with $\mu\ll E^2/2$; branch switch and collapse were observed there (Section~\ref{sec.stability}).}
\end{tabular}
\end{table}

\begin{figure}[htbp]
\centering
\includegraphics[width=.45\columnwidth]{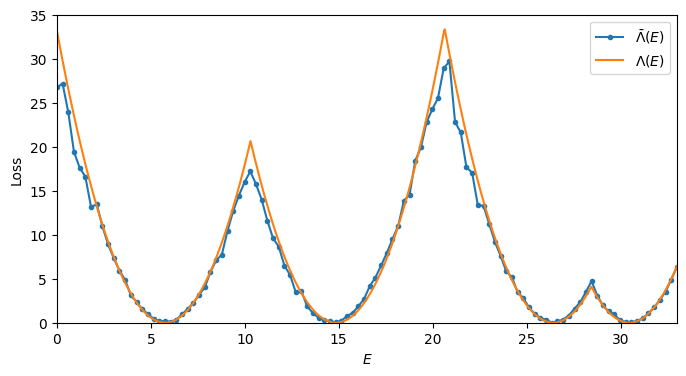}
\caption{The loss curve $\bar\Lambda(E)$ \rev{in the unit disk (Dirichlet Laplacian)} compared with the \rev{exact curve $\Lambda(E)=\min_k(E-E_k)^2$; near the peaks the penalized curve lies below $\Lambda(E)$, as predicted by Proposition~\ref{prop.penalized}}.}
\label{lossNN.fig}
\end{figure}

Fig.~\ref{lossNN.fig} shows the loss curve $\bar\Lambda(E^i)$ in the case where $\Omega$ is the unit disk in $\R^2$. \rev{The grid locates the eigenvalues up to its resolution; the value is then refined as described in Section~\ref{sec.certificate}.}
The associated computed eigenfunctions are displayed in Fig.~\ref{func.fig}\rev{.}

\begin{figure}[htbp]
\centering
\begin{tabular}{cccc}
\includegraphics[width=.16\columnwidth]{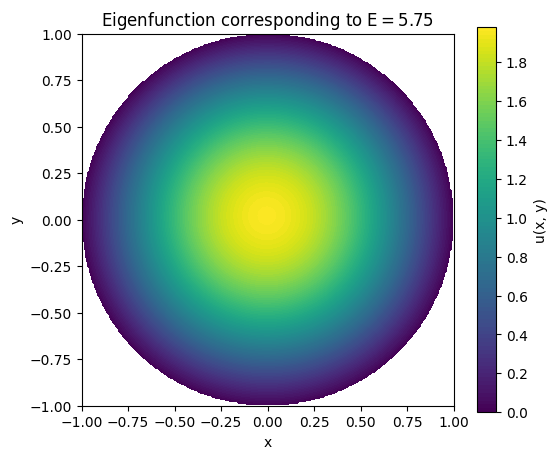} & \includegraphics[width=.16\columnwidth]{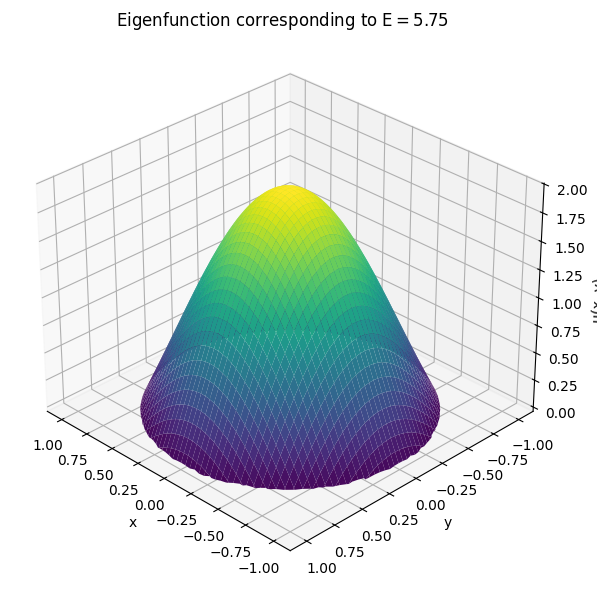} &
\includegraphics[width=.16\columnwidth]{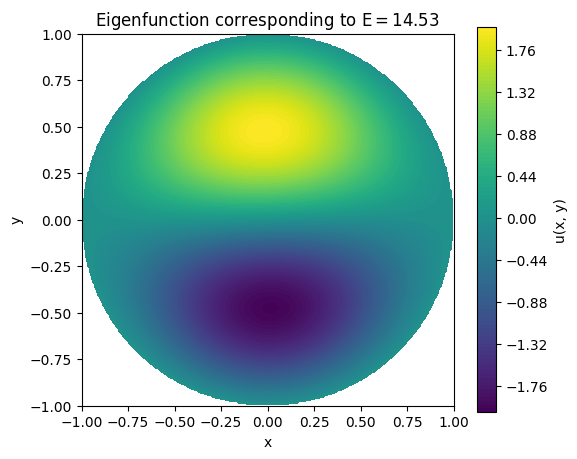} & \includegraphics[width=.16\columnwidth]{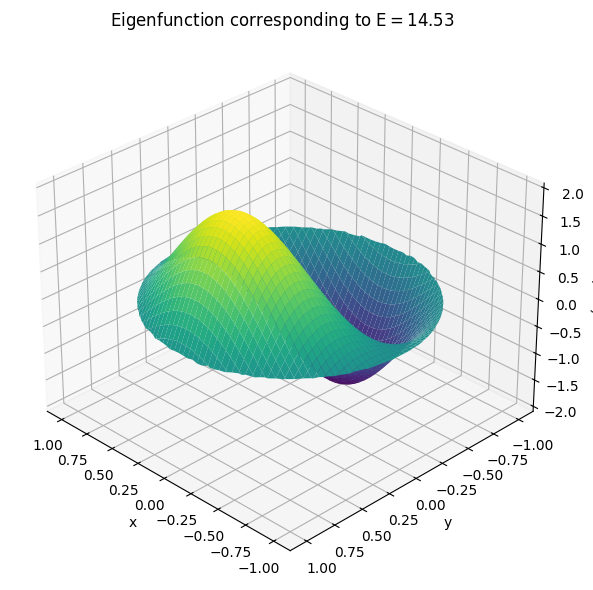} \\
\includegraphics[width=.16\columnwidth]{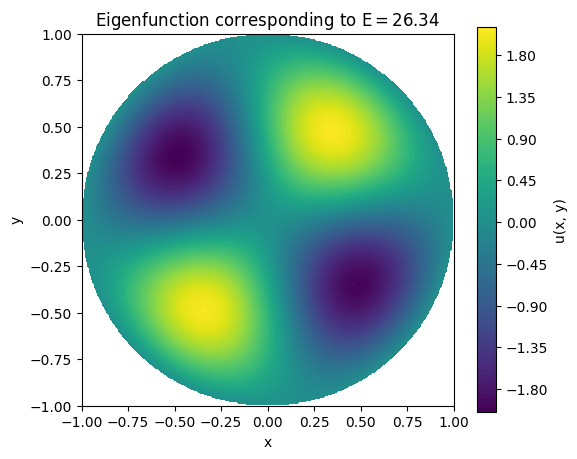} & \includegraphics[width=.16\columnwidth]{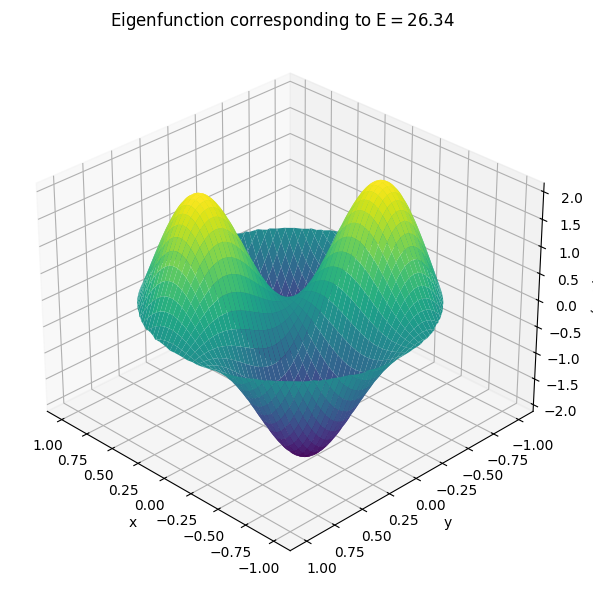} &
\includegraphics[width=.16\columnwidth]{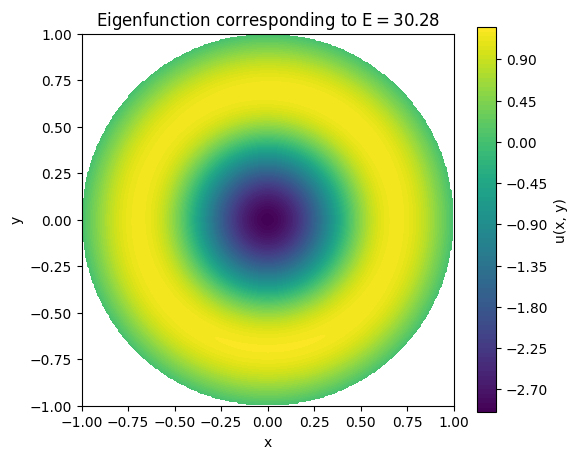} & \includegraphics[width=.16\columnwidth]{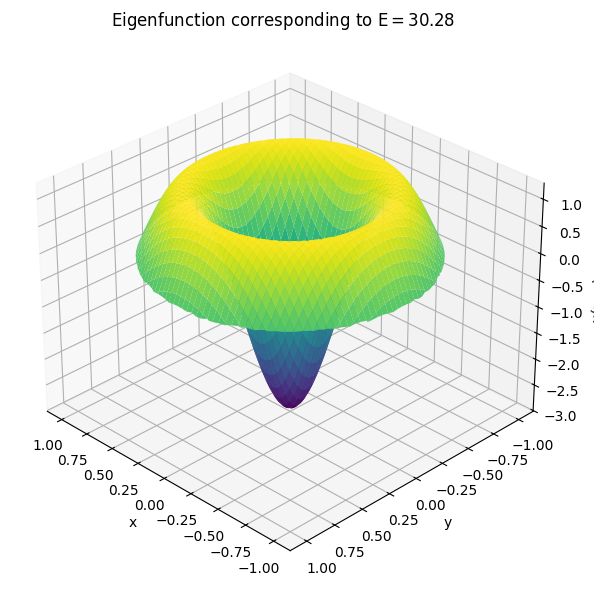} \\
\end{tabular}
\caption{The first four computed eigenfunctions of \eqref{eigen2} in the \rev{two-dimensional} unit disk.}
\label{func.fig}
\end{figure}

\subsection{Refinement and a posteriori certificate}\label{sec.certificate}
\begin{revblock}
Let $L=-\Delta+V$ and, for $u\in \mathrm{Dom}(L)\setminus\{0\}$ (for instance a trained network),
\[
r(u,E)=\frac{\|(L-E)u\|_2}{\|u\|_2}=\Lambda(u,E)^{1/2},\qquad
\rho(u)=\frac{\langle Lu,u\rangle}{\|u\|_2^2},\qquad \eta(u)=r(u,\rho(u)).
\]

\begin{proposition}\label{prop.rayleigh}
Let $u\in \mathrm{Dom}(L)\setminus\{0\}$. Then:
\begin{enumerate}
\item[(a)] $\Lambda(u,E)=(E-\rho(u))^2+\eta(u)^2$ for every $E\in\R$; in particular, for fixed $u$ the
minimum of $E\mapsto\Lambda(u,E)$ is attained at the Rayleigh quotient $\rho(u)$.
\item[(b)] (Weinstein \cite{Weinstein1934,Parlett1998}) $\ell(E)\le r(u,E)$ for every $E$: the interval
$[E-r,E+r]$ contains an eigenvalue of $L$.
\item[(c)] (Kato--Temple \cite{Kato1949,Parlett1998}) If $E_k$ is the only eigenvalue of $L$ in
$(\rho(u)-\gamma,\rho(u)+\gamma)$, then $|\rho(u)-E_k|\le\eta(u)^2/\gamma$.
\end{enumerate}
\end{proposition}

\noindent\emph{Refinement.} Once a local minimum $E^{i_0}$ of the discrete curve is detected, the
network $u^{i_0}$ is kept fixed and the eigenvalue is taken as the minimizer of
$E\mapsto\Lambda(u^{i_0},E)$, which by (a) is the Rayleigh quotient
\[
\hat E=\rho(u^{i_0})=-\frac{\int_\Omega u^{i_0}(\Delta u^{i_0}-Vu^{i_0})\,dx}{\int_\Omega (u^{i_0})^2\,dx}.
\]
No further training is needed. The same minimizer could be computed by a bounded scalar
minimization of the Monte Carlo objective re-sampled at every evaluation, but a stochastic objective
of this kind is flat to within its noise, and the returned value has an error of order the square
root of the noise level. Evaluating the closed form on a fixed sample is both cheaper and more
accurate: on 14 minima detected in the unit square the median absolute error was
$6.6\times10^{-3}$ with the re-sampled minimization, $1.1\times10^{-3}$ with the closed form on
$10^5$ fixed Monte Carlo points and $6.6\times10^{-4}$ on a $300\times300$ midpoint grid (maximal
errors $1.3\times10^{-1}$, $5.5\times10^{-3}$, $5.8\times10^{-3}$); a three-point parabolic fit of the
loss values at coarse nodes was not reliable (median $9.7\times10^{-3}$, maximum $1.0$), because
different nodes carry different training floors. All the eigenvalues reported in this paper are
computed with the closed form. We use the form $-\int u(\Delta u-Vu)/\int u^2$: on $4\times10^5$
Monte Carlo points its standard error was $2\times10^{-4}$--$2.4\times10^{-3}$ in the examples of
Section~\ref{sec.fem}, whereas the variational form $\int(|\nabla u|^2+Vu^2)/\int u^2$ had a standard
error $50$--$300$ times larger. Tables~\ref{table.2} and \ref{table.4} give the standard error of each
value; the quotients of Table~\ref{table.1} are evaluated on a midpoint grid, and no standard errors
were computed for them or for the scans of Section~\ref{sec.analysis}.

\noindent\emph{Certificate.} Items (b) and (c) hold for \emph{any} $u\in \mathrm{Dom}(L)$, whatever the quality
of the training. Hence each detected pair comes with the enclosure
$|E-E_k|\le r$ (with $r=\mathrm{Residual}^{1/2}$, \emph{not} the Residual itself), exact up to the Monte
Carlo error in $r$, and, if $\gamma$ is a
lower bound for the distance to the rest of the spectrum (obtainable from the enclosures of the
neighbouring detections), with $|\hat E-E_k|\le\eta^2/\gamma$. The eigenvalue error is therefore
\emph{quadratic} in the residual: a residual $\eta^2=10^{-1}$ and a gap $\gamma=10$ guarantee an
error at most $10^{-2}$, although $\eta\approx0.32$. This explains quantitatively the pattern
observed in all our tables, accurate eigenvalues together with comparatively large residuals. The
integrals in $r$ and $\rho$ are evaluated by Monte Carlo on fresh points, so the certificate is
statistical: the estimator of $\Lambda(u,E)$ has a bias $O(1/N)$ and a standard deviation
$O(N^{-1/2})$ that can be estimated from the same sample. With $4\times10^5$ points the standard
error of $r$ was $0.1$--$0.4\%$ of $r$ in the examples of Section~\ref{sec.fem}, negligible compared
with the margin by which the bounds hold (Section~\ref{sec.stability}).

\noindent\emph{Upper bound for $E_1$ and sign of the error.} A network function $u=B\,\mathrm{NN}$ is
smooth up to the boundary and vanishes on $\partial\Omega$, so integration by parts gives
$-\int u(\Delta u-Vu)=\int(|\nabla u|^2+Vu^2)$: with exact integration, $\rho(u)$ is the Rayleigh
quotient $R(u)$, and $\rho(u)\ge E_1$ by the Rayleigh--Ritz principle. The computed first eigenvalue is
therefore an upper bound for $E_1$, up to the Monte Carlo error of the integrals. If a lower bound
$\beta\le E_2$ with $\rho(u)<\beta$ is available, Temple's inequality \cite{Kato1949} gives the bracket
$\rho(u)-\eta(u)^2/(\beta-\rho(u))\le E_1\le\rho(u)$. For a higher eigenvalue, writing $u=\sum_ja_ju_j$,
\[
\rho(u)-E_k=\frac{\sum_ja_j^2(E_j-E_k)}{\sum_ja_j^2},
\]
which is quadratic in the components of $u$ outside the eigenspace of $E_k$ and positive when these
lie mainly in higher modes. Most of the differences we observed are positive (all those of
Tables~\ref{table.2} and \ref{table.4} and of the example in $\R^{10}$ of Section~\ref{sec.examples}),
which is consistent with this upward bias.
\end{revblock}

\noindent\rev{For the nonlinear problems of Section~\ref{sec.nolin} no analogue of the Weinstein or
Kato--Temple bounds is available; see Remark~\ref{rem.plap.lost}.}

\section{Further examples}\label{sec.examples}

In this section, we illustrate the flexibility of our method by applying it to a variety of settings beyond the two-dimensional unit disk with \rev{the hard-wall} potential. These examples include the computation of higher-order eigenvalues, problems in higher spatial dimensions, domains with different geometries, and more general potential functions.

\rev{The code is available in the repository given in the Declarations, together with the scripts
and raw results of the experiments of Tables~\ref{table.1}--\ref{table.fem} and
Sections~\ref{sec.nolin} and \ref{sec.analysis}.}

\subsection{Higher eigenvalues}
Our method enables the computation of higher-order eigenvalues without relying on sequential orthogonalization procedures. By choosing an appropriate interval $[E_*, E^*]$, \rev{the scan covers eigenvalues that are far from the principal one (in the examples below, up to the fourth eigenvalue $17\pi^2$ of the triangle). The windows below start at or below the first eigenvalue, or the network is warm-started at nodes close to the target; training from scratch at a high value of $E$ is unreliable (Section~\ref{sec.mu}).}

\rev{The choice of the penalty parameter $\mu$ in \eqref{lossmu} is discussed in Section~\ref{sec.mu}: along a warm-started scan $\mu$ must exceed the square of the largest distance from the window to the spectrum; it need not grow with $E$.}

To assess the accuracy of our method, we computed the eigenvalues of problem \eqref{eigen2} in the two-dimensional unit square $[0,1]\times[0,1]$ that lie within the interval \rev{$[44, 55]$}. In this case we have taken $\mu=3000$. In this range, problem \eqref{eigen2} has one \rev{(double)} eigenvalue, \rev{$E_{(1,2)}=5\pi^2\approx49.348$}. See Fig.~\ref{eigen55.fig} for the computed loss curve in this interval, together with the \rev{exact curve}.

\rev{We also computed the eigenvalues of \eqref{eigen2} in the unit square in the window $[0,105]$
(80 grid points, $\mu=500$, at most 5000 epochs, re-initialization every 25 grid points).} In Table \ref{table.1} we compare the exact eigenvalues with those obtained by our method, together with the relative error and residuals. \rev{The loss curve of this scan is shown in Fig.~\ref{fig.collapse} (Section~\ref{sec.stability}), and the first two eigenfunctions in Fig.~\ref{loss.lap.square.fig.2}.}

\begin{figure}[htbp]
\centering
\includegraphics[width=.45\columnwidth]{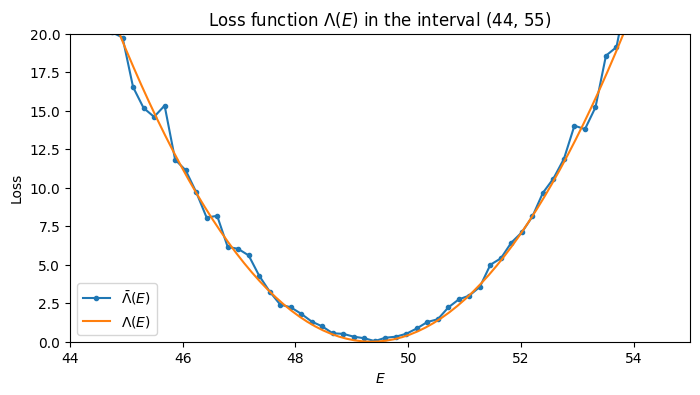}
\caption{The loss curve $\bar\Lambda(E)$ in the unit square computed in the interval \rev{$[44,55]$}.}
\label{eigen55.fig}
\end{figure}

\begin{table}[htbp]
\centering\small
\caption{Eigenvalues of problem \eqref{eigen2} in the unit square $[0,1]\times[0,1]$ compared with the exact values (scan of $[0,105]$ with 80 nodes, $\mu=500$, seed 0; $E$ is the Rayleigh quotient of the trained network evaluated on a $300\times300$ midpoint grid; Residual is evaluated on a fresh Monte Carlo sample at the reported $E$; $r=(\mathrm{Residual}/\overline{u^2})^{1/2}$ is the Weinstein radius). The double eigenvalues $5\pi^2$ and $10\pi^2$ appear as a single minimum; the multiplicity is not detected by the method.}
\label{table.1}
\begin{tabular}{ccccccc}
\toprule
$(m,n)$ & mult. & $E_{(m,n)}=\pi^2(m^2+n^2)$ & Our method & Rel.\ error & Residual & $r$ \\
\midrule
$(1,1)$ & 1 & $2\pi^2\approx 19.7392$ & 19.739909 & $3.5\times10^{-5}$ & $9.7\times10^{-2}$ & 0.31\\
$(1,2),(2,1)$ & 2 & $5\pi^2\approx 49.3480$ & 49.348167 & $2.9\times10^{-6}$ & $7.3\times10^{-2}$ & 0.27\\
$(2,2)$ & 1 & $8\pi^2\approx 78.9568$ & 78.956161 & $8.5\times10^{-6}$ & $2.0\times10^{-1}$ & 0.44\\
$(1,3),(3,1)$ & 2 & $10\pi^2\approx 98.6960$ & 98.695973 & $7.2\times10^{-7}$ & $9.3\times10^{-2}$ & 0.30\\
\bottomrule
\end{tabular}
\end{table}

\noindent\rev{All rows satisfy the Weinstein bound $|\hat E-E_k|\le r$, with errors two to four orders
of magnitude below $r$; the error of the Rayleigh quotient is quadratic in the residual
(Section~\ref{sec.certificate}).}

\begin{figure}[htbp]
\centering
\includegraphics[width=.2\columnwidth]{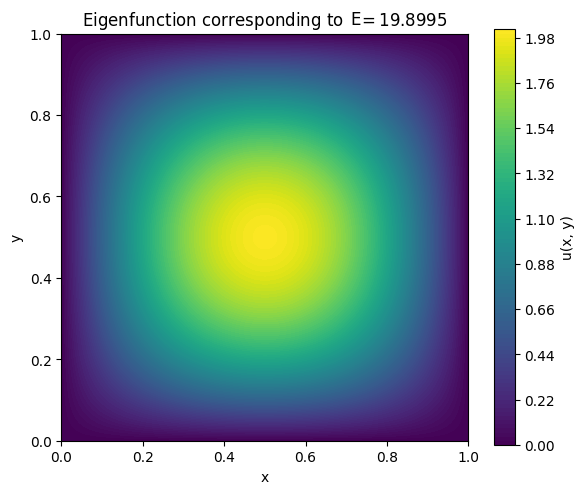}
\includegraphics[width=.2\columnwidth]{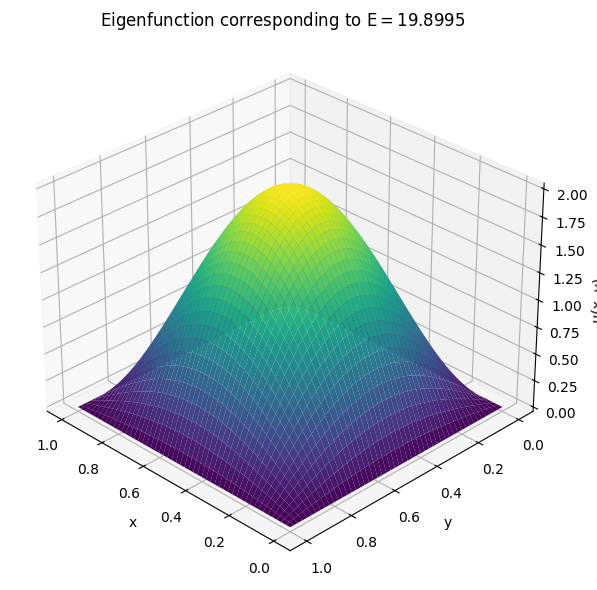}
\includegraphics[width=.2\columnwidth]{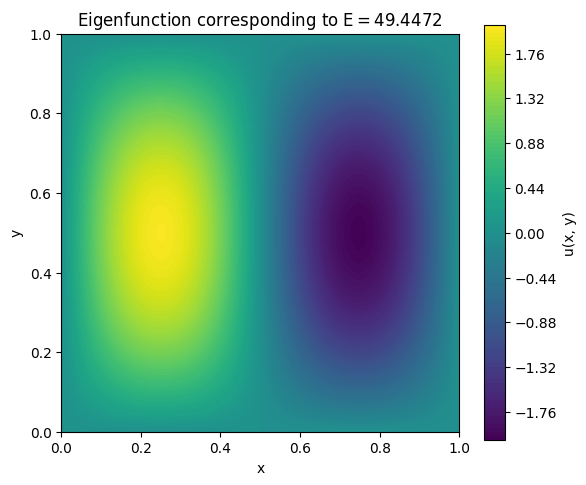}
\includegraphics[width=.2\columnwidth]{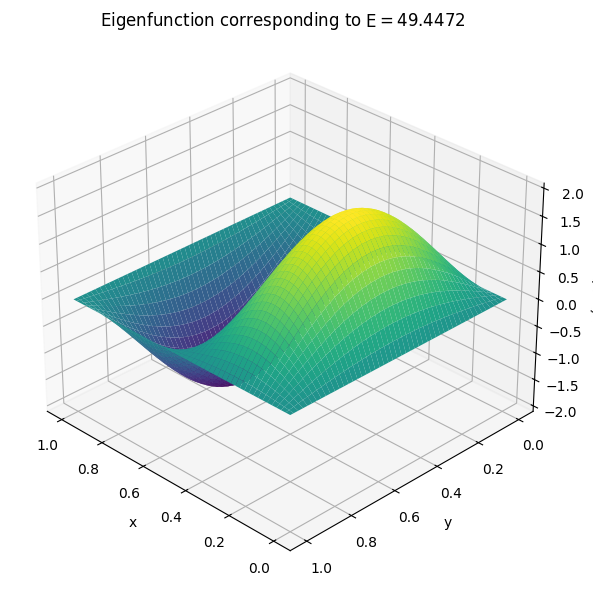}
\caption{The first two computed eigenfunctions in the unit square \rev{(from a run with $\mu=150$)}.}
\label{loss.lap.square.fig.2}
\end{figure}

\subsection{Higher dimensions}

The proposed framework extends naturally to higher spatial dimensions, with minimal changes required in the implementation. As an example, we consider the unit ball in $\R^d$, with $d=3,4$\rev{,} where the eigenvalue problem becomes:
$$
\begin{cases}
-\Delta u = Eu & \text{in } B^d\\
u=0 & \text{on }\partial B^d.
\end{cases}
$$
Despite the increased computational cost, the method remains effective, as shown in Fig.~\ref{loss3D.fig}. In this example we considered \rev{$\mu=100$ for $d=3$ and $\mu=200$ for $d=4$}.

\rev{The eigenvalues are known explicitly: for $m\in\N_0$, $n\in\N$, $E_{(m,n)}=j_{m+\frac12,n}^2$ with multiplicity $2m+1$ if $d=3$, and $E_{(m,n)}=j_{m+1,n}^2$ with multiplicity $(m+1)^2$ if $d=4$, where $j_{\nu,n}$ is the $n$-th positive zero of the Bessel function $J_\nu$.}

\begin{revblock}
Fig.~\ref{loss3D.fig} shows full scans starting at $E=0$. To measure the accuracy we trained, for
each of the first eigenvalues, on three nodes $E_k-0.9$, $E_k+0.1$, $E_k+1.1$ (fresh network at the
first node, warm start at the other two, seed 0, settings of Table~\ref{table.settings}) and report
the Rayleigh quotient at the node of smallest normalized loss (Table~\ref{table.2}). These windows
were placed using the known eigenvalues: they test the accuracy of the computed values, not their
detection. In $\R^3$ already the first (cold) node converged to the targeted eigenfunction
($|\hat E-E_k|\le0.01$, $r\approx1$). In $\R^4$ the quotient still follows the training value across
the window ($14.507$, $14.712$, $14.963$ at $E=13.78$, $14.78$, $15.78$): with 2000 epochs the network has
not fully converged to an eigenfunction, and the quotient is reliable only at the node of smallest
loss.
\end{revblock}
\begin{revblock}
The cost of an epoch grows only linearly with $d$ (Section~\ref{sec.cost}; we measured about $60+50d$
ms per epoch for the $3\times128$ network with 5000 points on one thread), whereas mesh-based methods
become infeasible in high dimension. As an illustration we scanned the unit ball of $\R^{10}$, whose
first eigenvalue $j_{4,1}^2=57.583$ is simple and whose second, $j_{5,1}^2=76.939$, has multiplicity
10, on the nodes $E=54,57,\dots,84$, a window chosen around the known eigenvalues (warm start, at most
2000 epochs, a single run with seed 0; 1.8 hours on one thread).
Most of the volume of this ball lies near its boundary, where $B=1-|x|^2$ is small; with $\mu=200$ the
network collapsed to $u\equiv0$ at the first node, and we used $\mu=2000$. The recorded residual has two
wells (19.8 at $E=57$ and 49.4 at $E=78$, against 234 at $E=72$), the tenfold eigenvalue appearing as a
single minimum. The Rayleigh quotient depends on the node: it is $57.62$--$57.84$ at the nodes of the
first well ($+0.07\%$ to $+0.45\%$) and $77.10$--$77.70$ at those of the second ($+0.2\%$ to $+1.0\%$); no
Monte Carlo standard error was computed for this run. With $r\approx4.7$--$7$ at the minima the
Weinstein enclosure is weak, so this example shows feasibility rather than accuracy.
\end{revblock}

\begin{table}[htbp]
\centering\small
\caption{Eigenvalues of \eqref{eigen2} in the unit balls of $\R^3$ and $\R^4$ (three-node windows, seed 0). $\hat E$: Rayleigh quotient at the node of smallest normalized loss, with its Monte Carlo standard error in the last digits ($4\times10^5$ points); $r$: Weinstein radius. In $\R^4$ the quotient drifts with the node ($14.51$, $14.71$, $14.96$ for the first eigenvalue); the value at the node of smallest loss is given. The higher eigenvalues in the windows of Fig.~\ref{loss3D.fig} were not recomputed in this way and are shown only in that figure.}
\label{table.2}
\begin{tabular}{ccccccc}
\toprule
$d$ & $(m,n)$ & mult. & $E_{(m,n)}$ & $\hat E$ & rel.\ error & $r$\\
\midrule
3 & $(0,1)$ & 1 & $j_{1/2,1}^2\approx9.8696$ & 9.8757(4) & $6.2\times10^{-4}$ & 0.41\\
3 & $(1,1)$ & 3 & $j_{3/2,1}^2\approx20.1907$ & 20.1912(4) & $2.2\times10^{-5}$ & 0.22\\
3 & $(2,1)$ & 5 & $j_{5/2,1}^2\approx33.2175$ & 33.2178(4) & $9\times10^{-6}$ & 0.26\\
3 & $(0,2)$ & 1 & $j_{1/2,2}^2\approx39.4784$ & 39.4789(6) & $1.3\times10^{-5}$ & 0.24\\
\midrule
4 & $(0,1)$ & 1 & $j_{1,1}^2\approx14.6820$ & 14.7122(5) & $2.1\times10^{-3}$ & 0.46\\
4 & $(1,1)$ & 4 & $j_{2,1}^2\approx26.3746$ & 26.3980(10) & $8.9\times10^{-4}$ & 0.79\\
\bottomrule
\end{tabular}
\end{table}

\begin{figure}[htbp]
\centering
\includegraphics[width=.4\columnwidth]{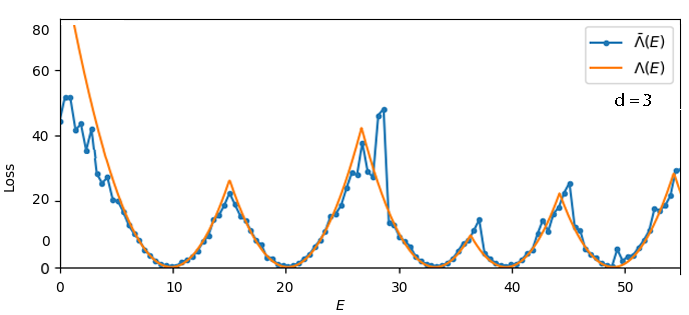}
\includegraphics[width=.4\columnwidth]{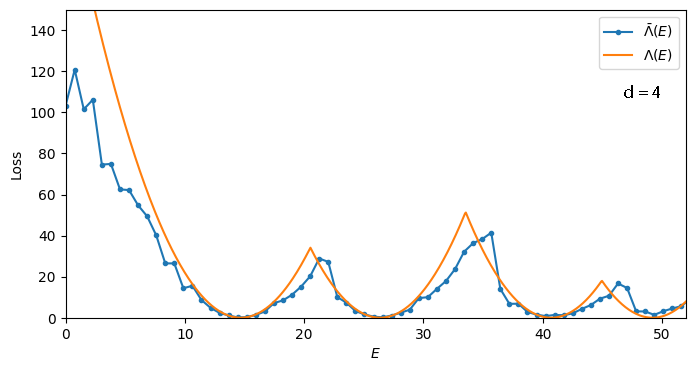}
\caption{The loss curve $\bar\Lambda(E)$ in the unit ball of $\R^3$ \rev{(left)} and of $\R^4$ \rev{(right)} compared with the \rev{exact curve}.}
\label{loss3D.fig}
\end{figure}

\subsection{Different geometries}
We also explore \rev{other geometries: squares, annuli and triangles}. The boundary factor $B(x)$ \rev{(Section~\ref{sec.implementation})} is adapted to these shapes, ensuring that the Dirichlet condition is enforced. In Fig.~\ref{loss.lap.square.fig.2}, we show eigenfunctions obtained in a square, in Fig.~\ref{loss.lap.annulus.fig}  the annular region is considered and in \rev{Figs.}~\ref{loss.lap.triang.fig} and \ref{loss.lap.triang.fig.2} the results in \rev{the triangle with vertices $(0,0)$, $(1,0)$, $(0,1)$} are shown. In that example we considered $\mu=15000$ \rev{(window $[0,180]$)}.

In Table \ref{table.4}, the exact eigenvalues of \eqref{eigen2} in \rev{this} triangle are compared with those obtained by our method. \rev{The Dirichlet eigenfunctions of the triangle are $\sin(m\pi x_1)\sin(n\pi x_2)-(-1)^{m+n}\sin(n\pi x_1)\sin(m\pi x_2)$, $m>n\ge1$, with eigenvalue $\pi^2(m^2+n^2)$; all the eigenvalues in the window are simple. The values were computed on three-node windows as for the balls. Here the first (cold) node of every window failed to converge within 2000 epochs at $\mu=15000$ ($\hat E=56.0$, $57.1$, $88.0$, $164.7$ with $r=15$--$62$, i.e.\ a wrong branch), and the values come from the second or third node; the residuals are larger than in the full warm-started scan of Fig.~\ref{loss.lap.triang.fig}, but all rows satisfy the Weinstein bound.}

\begin{table}[htbp]
\centering\small
\caption{Eigenvalues of problem \eqref{eigen2} in the triangle with vertices $(0,0)$, $(1,0)$, $(0,1)$ compared with the exact values (three-node windows, $\mu=15000$, seed 0). All the eigenvalues listed are simple; these are all the eigenvalues in $[0,180]$. $\hat E$ with Monte Carlo standard error; $r$: Weinstein radius.}
\label{table.4}
\begin{tabular}{ccccc}
\toprule
$(m,n)$ & $E_{(m,n)} = \pi^2 (m^2 + n^2)$ & $\hat E$ & rel.\ error & $r$\\
\midrule
$(2,1)$ & $5\pi^2\approx 49.3480$ & 49.3509(14) & $5.8\times10^{-5}$ & 0.94\\
$(3,1)$ & $10\pi^2\approx 98.6960$ & 98.7067(24) & $1.1\times10^{-4}$ & 1.81\\
$(3,2)$ & $13\pi^2\approx 128.3049$ & 128.4612(95) & $1.2\times10^{-3}$ & 5.68\\
$(4,1)$ & $17\pi^2\approx 167.7833$ & 167.7923(35) & $5.4\times10^{-5}$ & 2.51\\
\bottomrule
\end{tabular}
\end{table}

\begin{figure}[htbp]
\centering
\includegraphics[width=.26\columnwidth]{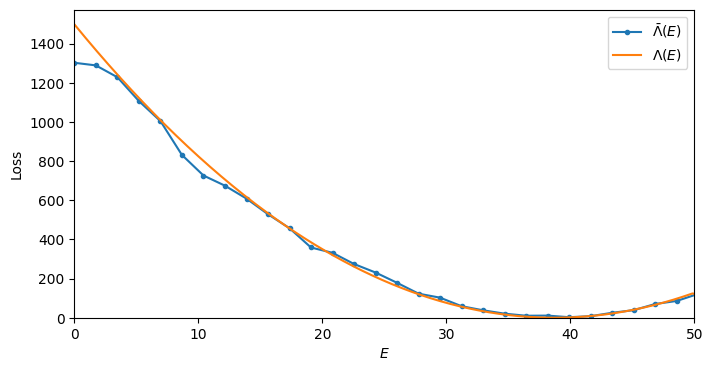}
\includegraphics[width=.26\columnwidth]{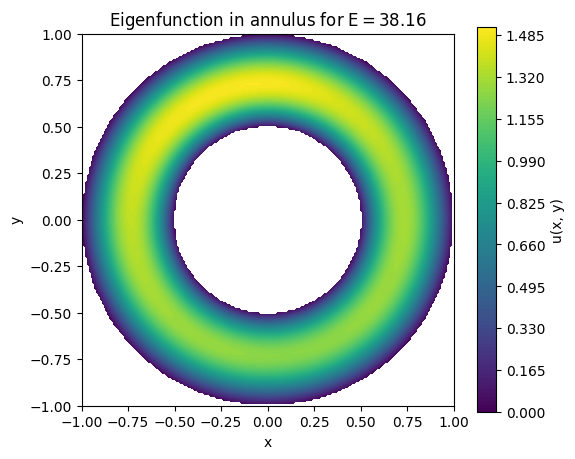}
\includegraphics[width=.26\columnwidth]{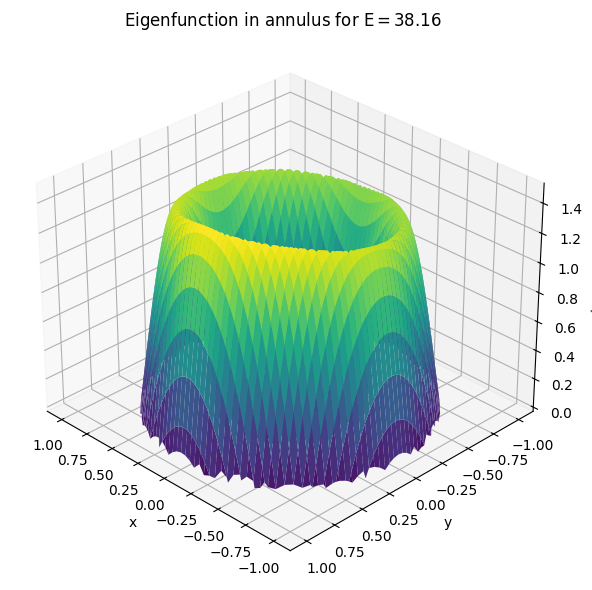}
\caption{\rev{Left: the loss curve $\bar\Lambda(E)$ in the annulus $\{0.5<|x|<1\}$; the reference curve drawn in this plot contains only the first eigenvalue $39.0133$, whereas the window also contains $40.8725$ (double) and $46.4285$ (double). The PINN scan followed the radial branch and missed these two eigenvalues (an example of a missed eigenvalue, see Section~\ref{sec.stability}). Centre and right: the first computed eigenfunction.}}
\label{loss.lap.annulus.fig}
\end{figure}

\begin{figure}[htbp]
\centering
\includegraphics[width=.5\columnwidth]{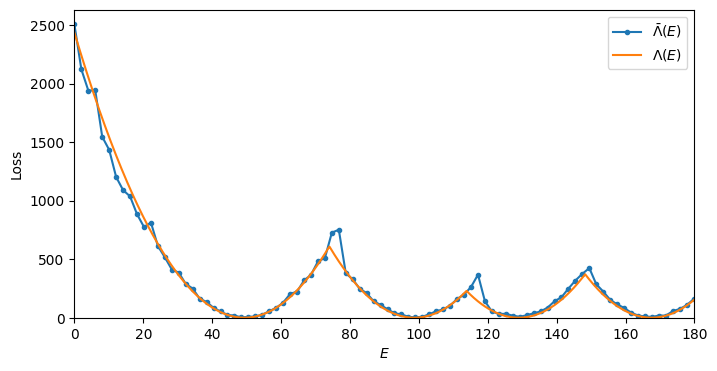}
\caption{The loss curve $\bar\Lambda(E)$ in the triangle \rev{with} vertices $(0,0)$, $(0,1)$, $(1,0)$ compared with the \rev{exact curve}.}
\label{loss.lap.triang.fig}
\end{figure}

\begin{figure}[htbp]
\centering
\includegraphics[width=.2\columnwidth]{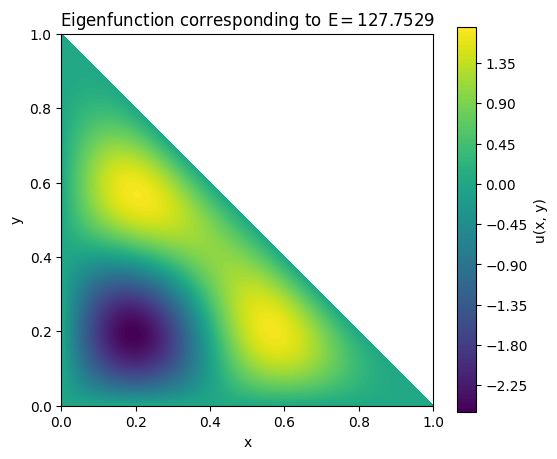}
\includegraphics[width=.2\columnwidth]{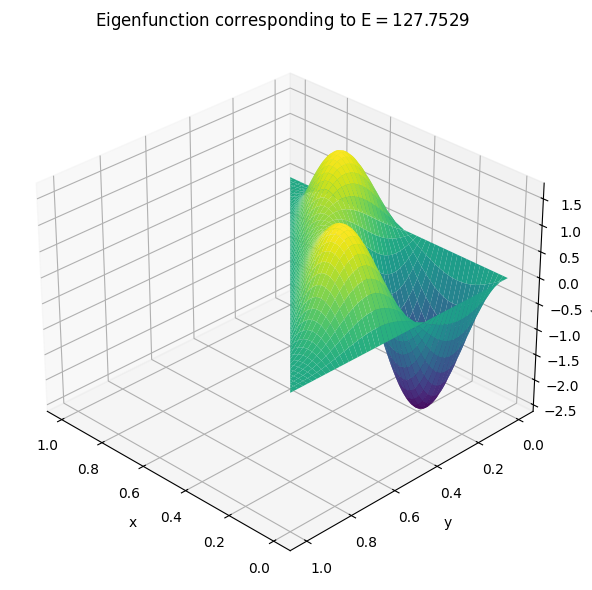}
\includegraphics[width=.2\columnwidth]{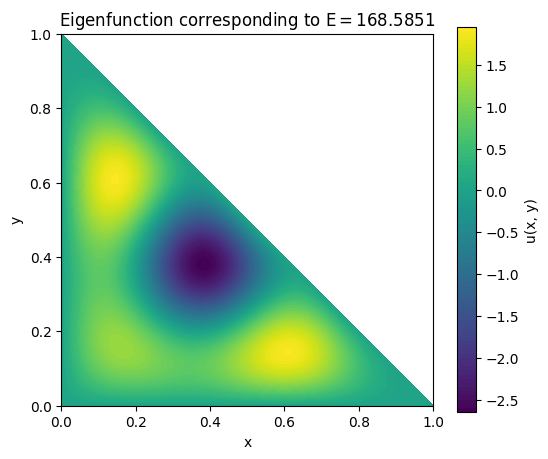}
\includegraphics[width=.2\columnwidth]{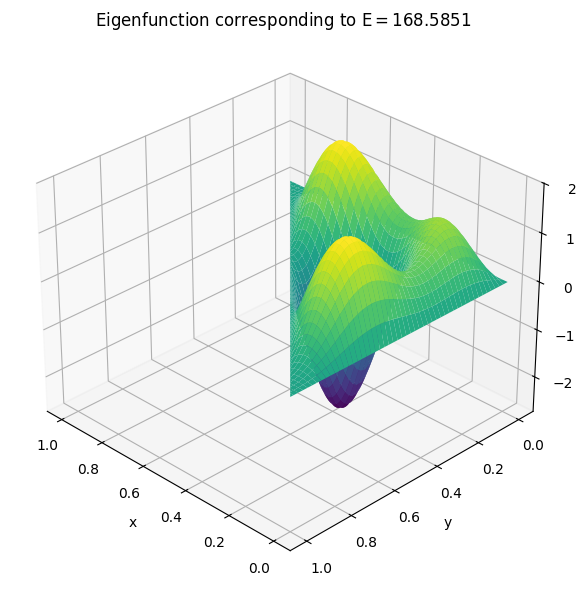}
\caption{Third and fourth computed eigenfunctions in the triangle \rev{with} vertices $(0,0)$, $(0,1)$, $(1,0)$ \rev{(from a different run than Table~\ref{table.4}; the panel titles show the values of that run)}.}
\label{loss.lap.triang.fig.2}
\end{figure}

\subsection{Different potentials}\label{sec.potentials}

Beyond hard-wall confinement, our approach accommodates general potential functions $V(x)$\rev{; the problem is always solved in the unit disk or ball with Dirichlet boundary condition}.

As an example, we consider a harmonic potential $V(x)=\frac{\omega^2}{2}|x|^2$ \rev{in the unit disk} for different values of $\omega$\rev{; note that, because of the boundary condition, this is not the harmonic oscillator on $\R^2$}. Fig.~\ref{pot2.fig} displays \rev{the} loss curve and the first eigenfunction for $\omega=1$ and Fig.~\ref{pot1.fig} for $\omega=10$. In Fig.~\ref{fig.well-two}, we display the double-well Gaussian potential \rev{in the unit ball of $\R^3$,
\[
V(x)=-V_0\Big(e^{-|x-ae_1|^2/(2\sigma^2)}+e^{-|x+ae_1|^2/(2\sigma^2)}\Big),\qquad a=1,\ \sigma=0.5,\ V_0=1,
\]
i.e.\ $V(x)=-e^{-2|x-x_0|^2}-e^{-2|x+x_0|^2}$ with $x_0=e_1=(1,0,0)$; the centres of the wells lie on the boundary sphere,} together with the loss curve.

Finally, in Fig.~\ref{fig.well-four} we consider a four-well compactly supported potential \rev{in the unit ball of $\R^3$} and we plot the corresponding loss curve. More precisely, \rev{for $x\in\R^3$ let} the centers of the four wells \rev{be} $z_1 = (a,a,0)$, $z_2 = (-a,a,0)$, $z_3 = (a,-a,0)$, $z_4 = (-a,-a,0)$, where $a=0.22$, \rev{let $\sigma=0.2$, and} define the compactly supported bump function
$$
\phi(t) =
\begin{cases}
\exp(-(1-t^2)^{-1}), & t < 1,\\
0, & t \ge 1.
\end{cases}
$$
Then the potential is
$
V(x) = - \sum_{i=1}^{4} \phi\big( |x - z_i|/\sigma \big)
$
\rev{(depth $e^{-1}\approx0.37$ at each centre)}.

\rev{In the harmonic, double-well and four-well examples we used $\mu=100$, $400$ and $1000$, respectively (2000 sampling points for the four-well potential).}

\begin{figure}[htbp]
\centering
\includegraphics[width=.26\columnwidth]{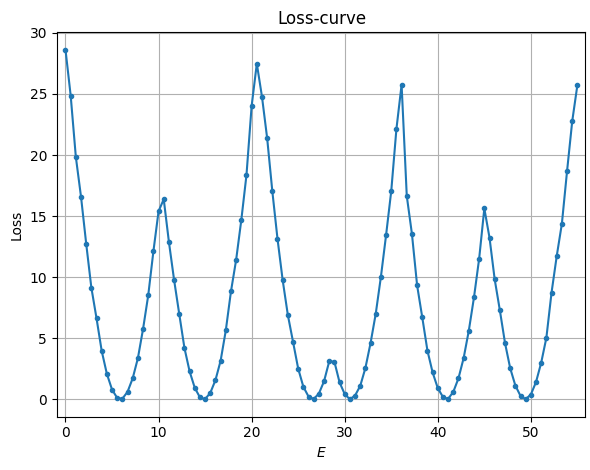}
\includegraphics[width=.26\columnwidth]{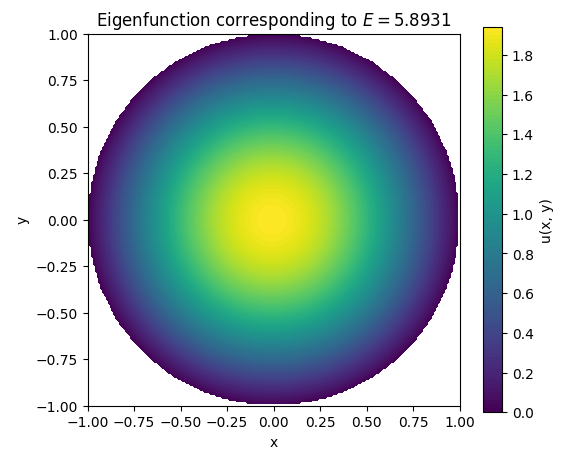}
\includegraphics[width=.26\columnwidth]{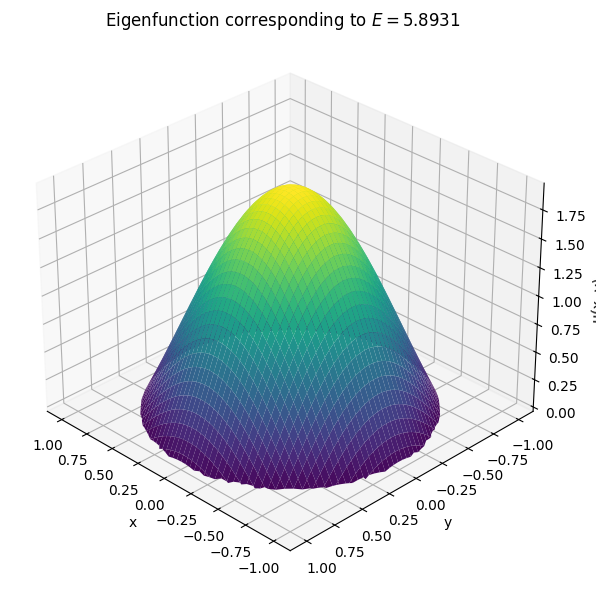}
\caption{The loss curve $\bar\Lambda(E)$ in the unit disk with harmonic potential $V=\frac{\omega^2}{2}|x|^2$ and $\omega=1$ with the first computed eigenfunction\rev{. The finite element eigenvalues in the window (Section~\ref{sec.fem}) are $5.8919$, $14.8484$ (double), $26.5790$ (double), $30.6271$, $40.9385$ (double), $49.3852$ (double); the minima of this curve were not refined}.}
\label{pot2.fig}
\end{figure}

\begin{figure}[htbp]
\centering
\includegraphics[width=.26\columnwidth]{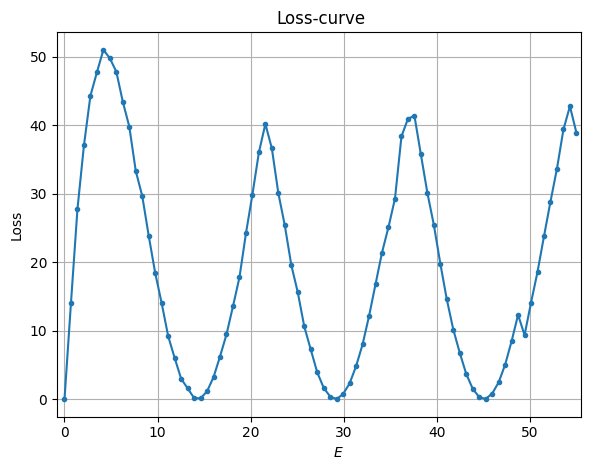}
\includegraphics[width=.26\columnwidth]{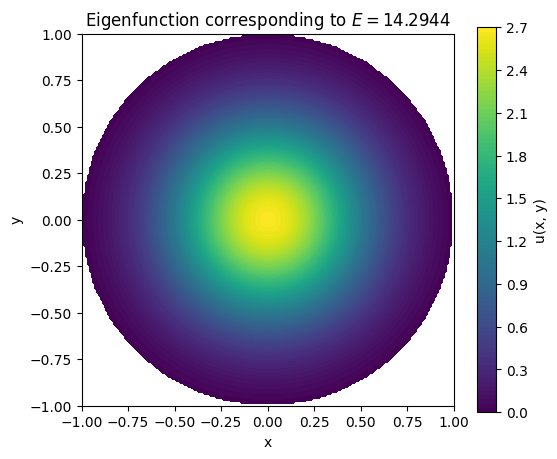}
\includegraphics[width=.26\columnwidth]{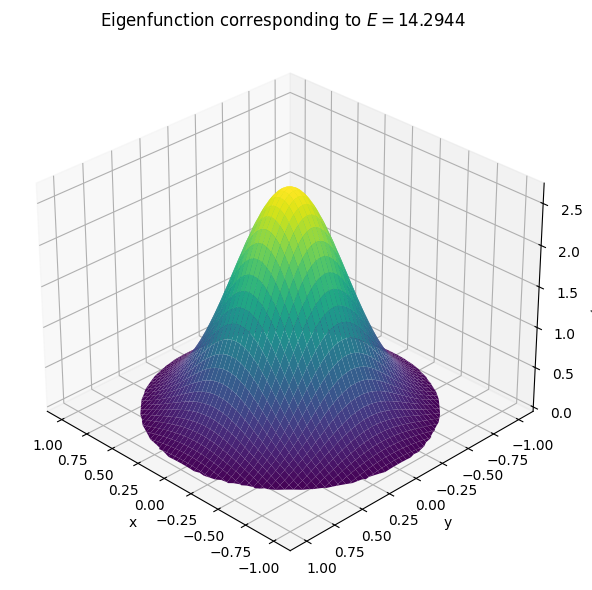}
\caption{The loss curve $\bar\Lambda(E)$ in the unit disk with harmonic potential $V=\frac{\omega^2}{2}|x|^2$ and $\omega=10$ with the first computed eigenfunction\rev{. The eigenvalues in the window are $14.2847$, $29.1349$ (double), $45.0887$ (double), $46.2907$; the last two are $1.2$ apart and appear as a single minimum in this scan (grid step $\approx0.7$; with step $0.5$ they are separated, Section~\ref{sec.fem}). The value $\bar\Lambda=0$ at $E=0$ is the trivial state $u=0$ ($\ell(0)^2=204>2\mu=200$, Proposition~\ref{prop.penalized})}.}
\label{pot1.fig}
\end{figure}

\begin{figure}[htbp]
\centering
\includegraphics[width=.24\columnwidth]{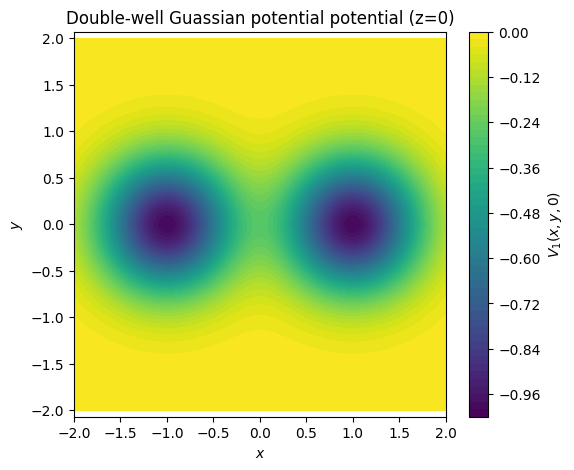}
\includegraphics[width=.24\columnwidth]{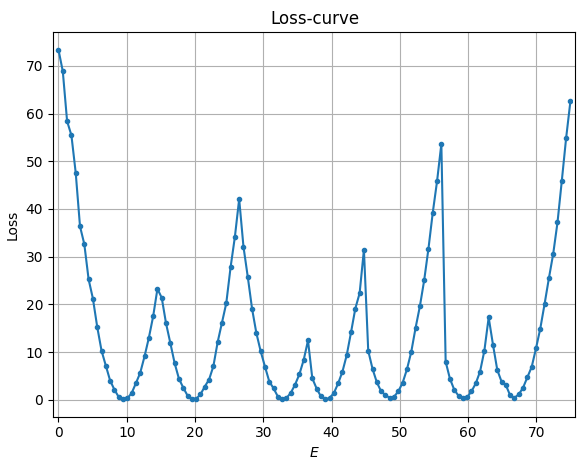}
\includegraphics[width=.24\columnwidth]{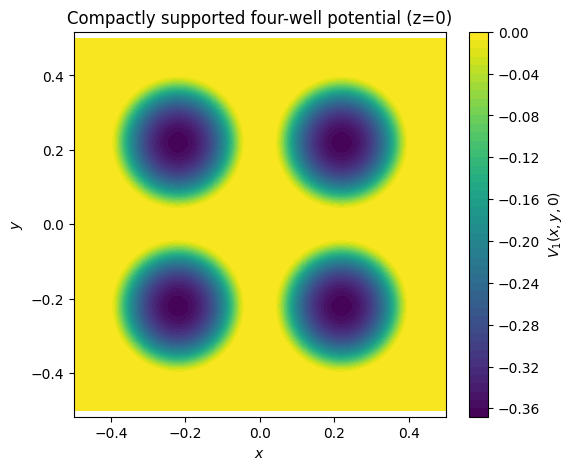}
\includegraphics[width=.24\columnwidth]{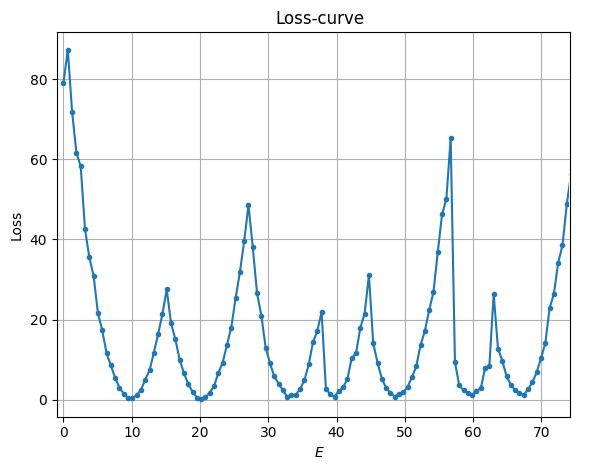}
\caption{\rev{Unit ball of $\R^3$. Left half: double-well Gaussian potential and loss curve $\bar\Lambda(E)$; right half: compactly supported four-well potential and loss curve.}}
\label{fig.well-two}\label{fig.well-four}
\end{figure}

\subsection{Validation with the finite element method}\label{sec.fem}
\begin{revblock}
For the potentials no closed-form eigenvalues are available, so we computed independent reference
values with a finite element method: Lagrange $P_2$ elements on curved (isoparametric) meshes,
generalized eigenvalue problem solved by shift-invert Lanczos, Richardson extrapolation over the
last three uniform refinements. The finite element code reproduces the exact eigenvalues of the
disk, annulus, square and triangle with errors $10^{-10}$--$10^{-6}$ (observed order $4$ in two
dimensions) and those of the ball of $\R^3$ with errors $7\times10^{-5}$ to $1.6\times10^{-2}$
for $E\le67$ (observed order $\approx2.8$, three mesh levels with up to $41\,727$ degrees of
freedom); the reported uncertainty of the extrapolated values is conservative on these test cases.

For each potential we then trained the PINN on four windows of seven grid points (step $0.5$),
with the settings of Table~\ref{table.settings} and 2--3 seeds; the eigenvalue is the Rayleigh
quotient of the network at the node of smallest normalized loss. \emph{The windows were centred at
the finite element values}: this is a test of the accuracy of the loss curve at known eigenvalues,
not a blind detection; a blind scan is reported below. Table~\ref{table.fem} and Fig.~\ref{fig.fem}
show the results. Since the reference values are themselves approximations, we speak of relative
differences.

\begin{table}[tbp]
\centering\footnotesize
\caption{PINN versus finite elements (FEM) for the potentials of Section~\ref{sec.potentials}. FEM: extrapolated value $\pm$ estimated uncertainty (multiplicity); $^\dagger$: other FEM eigenvalues within $0.75$, i.e.\ a cluster that the scan does not resolve (the reference is then the cluster member nearest to the computed value). $\hat E$: Rayleigh quotient of the trained network for each seed. Rel.\ diff.: of the seed mean, and largest over single seeds.}
\label{table.fem}
\begin{tabular}{@{}lllccc@{}}
\toprule
example & FEM (mult.) & $\hat E$ per seed & \multicolumn{2}{c}{rel.\ diff.\ (\%)} & $\min\Lambda$\\
 & & & mean & worst & \\
\midrule
double well, ball $\R^3$ & $9.5764\pm0.0005$ (1) & 9.5783, 9.5782 & $+0.019$ & 0.019 & 0.065\\
 & $19.7522\pm0.003$ (1)$^\dagger$ & 19.7640, 19.7547 & $+0.036$ & 0.060 & 0.077\\
 & $32.8735\pm0.02$ (2)$^\dagger$ & 32.9552, 32.8571 & $+0.099$ & 0.248 & 0.117\\
 & $39.1924\pm0.01$ (1) & 39.1944, 39.1938 & $+0.004$ & 0.005 & 0.108\\
\midrule
four well, ball $\R^3$ & $9.8542\pm0.0005$ (1) & 9.8593, 9.8646, 9.8581 & $+0.066$ & 0.105 & 0.355\\
 & $20.1801\pm0.003$ (2)$^\dagger$ & 20.1505, 20.1878, 20.1870 & $-0.025$ & 0.147 & 0.478\\
 & $33.2104\pm0.01$ (2)$^\dagger$ & 33.2141, 33.2077, 33.2061 & $-0.003$ & 0.013 & 0.407\\
 & $39.4573\pm0.01$ (1) & 39.4578, 39.4414, 39.4554 & $-0.015$ & 0.040 & 1.054\\
\midrule
harmonic $\omega=10$, disk & $14.2847$ (1) & 14.2899, 14.2904 & $+0.038$ & 0.040 & 0.174\\
 & $29.1349$ (2) & 29.1357, 29.1353 & $+0.002$ & 0.003 & 0.056\\
 & $45.0887$ (2) & 45.0892, 45.0892 & $+0.001$ & 0.001 & 0.065\\
 & $62.5793$ (2) & 62.5792, 62.5800 & $+0.0005$ & 0.001 & 0.076\\
\bottomrule
\multicolumn{6}{@{}l}{FEM uncertainties for $\omega=10$ (two dimensions): $3\times10^{-7}$--$2\times10^{-5}$.}
\end{tabular}
\end{table}

\begin{figure}[tbp]
\centering
\includegraphics[width=.62\textwidth]{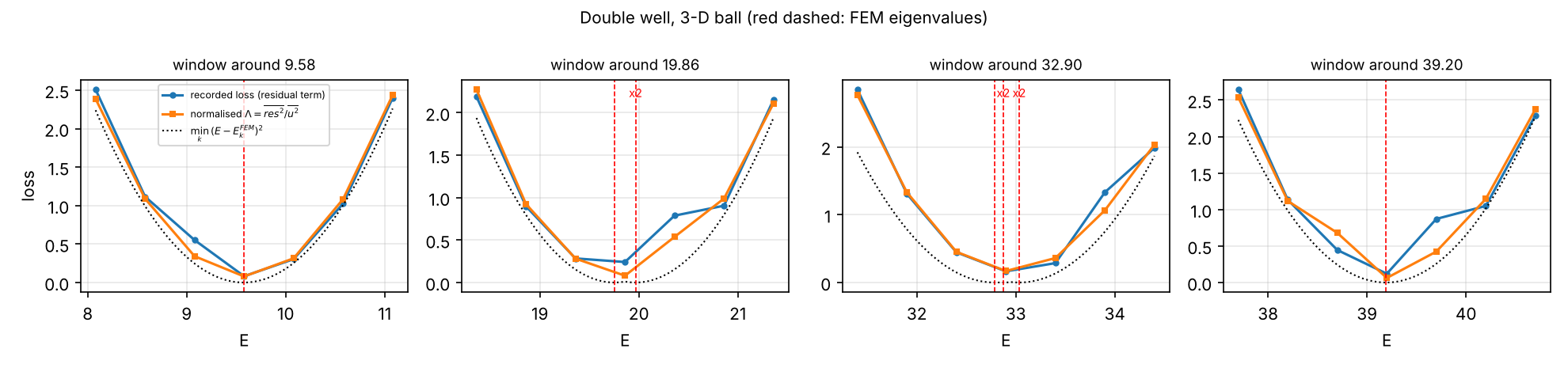}\\
\includegraphics[width=.62\textwidth]{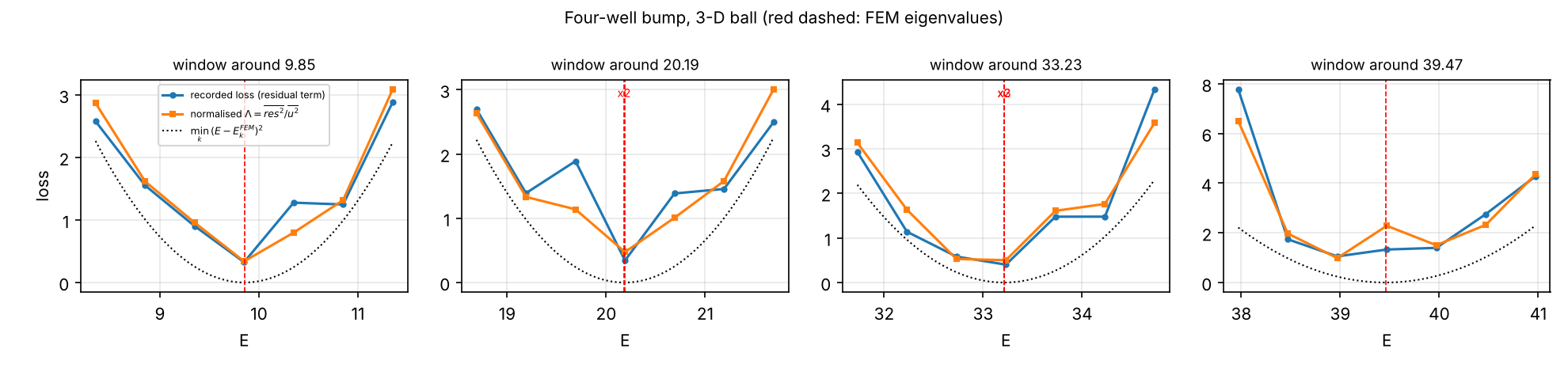}\\
\includegraphics[width=.62\textwidth]{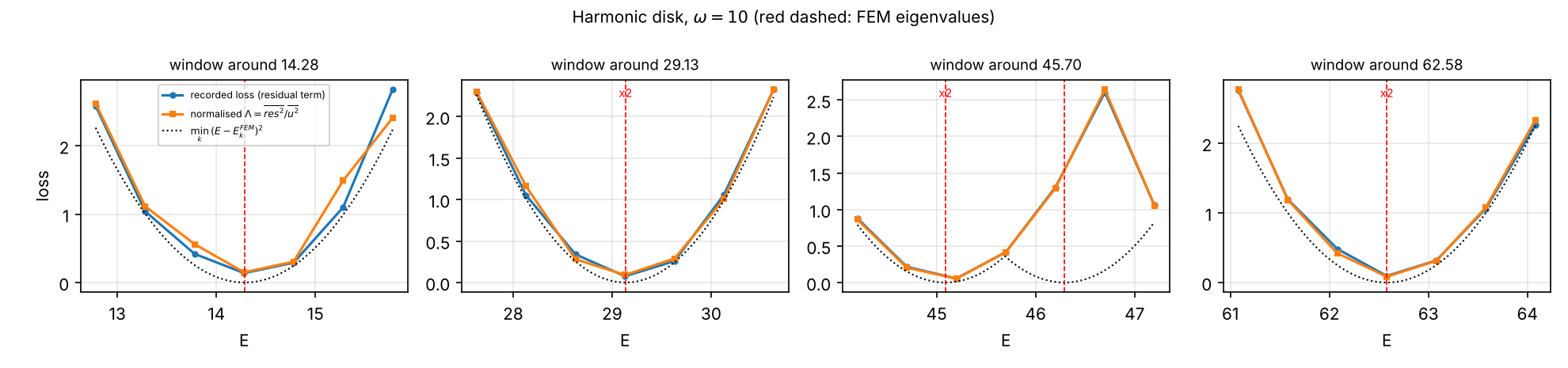}
\caption{PINN loss curves on windows centred at finite element eigenvalues (red dashed lines; ``x2'': double) for the double-well (top), four-well (middle) and harmonic $\omega=10$ (bottom) potentials: recorded residual term, normalized loss $\Lambda$, and $\min_k(E-E_k^{\rm FEM})^2$ (dotted).}
\label{fig.fem}
\end{figure}

The findings are the following. (i) The seed means agree with the finite element values to within
$0.1\%$ and single runs to within $0.25\%$; the largest single-run difference ($+0.25\%$, double well
near $32.9$) occurs in a window containing three eigenvalues spread over $0.25$, and lies inside the
cluster; outside clusters the largest difference is $0.105\%$ (four well, $9.85$). In three dimensions
the finite element uncertainty ($5\times10^{-4}$ to $2\times10^{-2}$) is comparable to the
differences, and in three of the eight windows it exceeds them. The minimum of the loss is not zero
($\min\Lambda=0.06$--$1.05$, i.e.\ $r=0.25$--$1$): as explained in Section~\ref{sec.certificate}, the
accuracy comes from the Rayleigh quotient, whose error is quadratic in $r$. (ii) The spectra of the
well potentials consist of \emph{clusters of close eigenvalues}: for the double well $19.75$ and
$19.97$ (double), and $32.79$, $32.87$ (double), $33.03$ (double); for the four well $20.180$ (double)
and $20.189$, and $33.210$ (double) and $33.214$ (triple). The loss curve has one minimum per cluster
and the quotient returns a value inside the cluster; the number of eigenvalues in the cluster and
their multiplicities cannot be read from the curve (see Remark~\ref{rem.stability}). A blind scan of
the double well on $[4.7,49.7]$ (step $1.5$, grid not aligned with the spectrum, seed 0) confirms this
(Table~\ref{table.blind}): the interval contains $11$ distinct finite element eigenvalues ($17$ with
multiplicity) in five clusters; the loss curve has exactly five local minima, one per cluster, and
every finite element eigenvalue lies within $0.75$ of one of them. (iii) The well potentials are weak
perturbations of the free ball: the first eigenvalue moves from $\pi^2=9.8696$ to $9.5764$ (double
well) and to $9.8542$ (four well); for the four-well potential the shift is of the same order as the
PINN differences, so this example hardly distinguishes the potential from $V=0$. (iv) In the harmonic
case with $\omega=10$, the double eigenvalue $45.0887$ and the simple eigenvalue $46.2907$ are $1.2$
apart. With warm start the network stays on the branch of $45.09$ up to $E\approx46.2$--$46.7$ and then
switches to $46.29$ (quotient $46.289$--$46.291$ at the last nodes of both seeds). On a window
extended to $49.2$ (step $0.5$, seed 0) the normalized loss has two separate minima, at $E=45.2$
($\hat E=45.093$) and at $E=46.2$ ($\hat E=46.2903$), separated by a ridge at $45.7$; with the coarser
step $\approx0.7$ of Fig.~\ref{pot1.fig} this hysteresis hides the second eigenvalue. (v) For these
linear problems the finite element method is much cheaper and more accurate: a 7-point PINN window
took $7$--$25$ minutes on one CPU core, whereas the finite element computation of all eigenvalues in
$[0,75]$ with multiplicities took less than $10$\,s in two dimensions and about $2$ minutes for the
ball of $\R^3$, also on one core. As stated in Section~\ref{sec.approach}, the purpose of the method is
not to compete with finite elements on linear problems in low dimension.

\begin{table}[tbp]
\centering\small
\caption{Blind scan of the double well (unit ball of $\R^3$, $E=4.7,6.2,\dots,49.7$, $\mu=400$, seed 0): local minima of the recorded loss, Rayleigh quotient $\hat E$ of the network at the minimum, finite element eigenvalues within $1.5$ (multiplicity), and relative difference of $\hat E$ to the nearest one.}
\label{table.blind}
\begin{tabular}{cccc}
\toprule
minimum at $E$ & $\hat E$ & FEM eigenvalues within $1.5$ & rel.\ diff.\ (\%)\\
\midrule
9.2 & 9.5779 & 9.576 (1) & $+0.015$\\
19.7 & 19.7545 & 19.752 (1), 19.967 (2) & $+0.012$\\
33.2 & 32.7966 & 32.786 (1), 32.874 (2), 33.034 (2) & $+0.034$\\
39.2 & 39.1945 & 39.192 (1) & $+0.005$\\
48.2 & 48.5270 & 48.406 (1), 48.430 (2), 48.545 (2), 48.676 (2) & $-0.038$\\
\bottomrule
\end{tabular}
\end{table}
\end{revblock}

\section{Nonlinear problems} \label{sec.nolin}

Our method can also be extended to nonlinear eigenvalue problems. As an illustrative example, we consider the eigenvalue problem for the $p$-Laplacian:
\begin{equation}\label{p-eigen}
\begin{cases}
-\nabla \cdot (|\nabla u|^{p-2} \nabla u) = E |u|^{p-2} u & \text{in } \rev{\Omega}, \\
u = 0 & \text{on } \partial \rev{\Omega},
\end{cases}
\end{equation}
where $\rev{\Omega} \subset \R^{\rev{d}}$ is \rev{a bounded domain} and $p > 1$.

To approximate the solution using our framework, we modify the loss function accordingly. For a fixed $E$, we define the following loss:
\begin{equation}\label{loss.p}
\bar\Lambda_p(u, E) = \int_{\rev{\Omega}} \left( \Delta_p u + E |u|^{p-2} u \right)^2 \, dx + \mu \left( \int_{\rev{\Omega}} |u|^p\, dx - 1 \right)^2,
\end{equation}
where $\Delta_p u = \nabla \cdot (|\nabla u|^{p-2} \nabla u)$\rev{, and the loss curve $\Lambda_p(E)=\inf_u\bar\Lambda_p(u,E)$}.

\begin{revblock}
The loss \eqref{loss.p} is homogeneous of the right degree: since $\Delta_p(tu)=|t|^{p-2}t\,\Delta_pu$,
the residual of $tu$ is $|t|^{p-2}t$ times the residual of $u$, so the normalization $\int|u|^p=1$
fixes only the scale and not the eigenvalue. There is no spectral decomposition, and the identity
\eqref{cota.igual} fails; what survives is shown in Proposition~\ref{prop.plap} below: the zeros
of the loss curve are exactly the eigenvalues, and the curve satisfies the upper bound
\[
\Lambda_p(E)\le\inf_kc_k(E-E_k)^2,\qquad c_k=\int_\Omega|u_k|^{2(p-1)}dx,
\]
so that it lies below a parabola $c_k(E-E_k)^2$ near each eigenvalue $E_k$ and every isolated eigenvalue
produces a well. No lower bound of the form $\Lambda_p(E)\ge c\,\dist(E,\sigma_p)^2$ (an analogue of
the Weinstein bound) is available.

\subsection{What can be proved}
For $p\ne2$ the eigenvalues of \eqref{p-eigen} form a closed set $\sigma_p\subset(0,\infty)$; the
first eigenvalue is simple and isolated \cite{Anane1987,Lindqvist1990}, there is a sequence of
variational (Lusternik--Schnirelmann) eigenvalues \cite{GarciaAzorero-Peral}, the second
eigenvalue coincides with the second variational one \cite{Anane-Tsouli} and has a mountain-pass
characterization \cite{Cuesta-deFigueiredo-Gossez}, which is the basis of the algorithm of
\cite{HJ11}. It is not known whether the variational eigenvalues exhaust $\sigma_p$. Eigenfunctions
are bounded and of class $C^{1,\alpha}$ \cite{DiBenedetto,Tolksdorf}, but in general not $C^2$.

\begin{proposition}\label{prop.plap}
Let $p>1$ with $p\ge\frac{2d}{d+2}$, and let $\mathcal D_p=\{u\in W^{1,p}_0(\Omega)\cap
L^{2(p-1)}(\Omega)\colon\Delta_pu\in L^2(\Omega)\}$. Then:
\begin{enumerate}
\item[(a)] If $u_k$ is an eigenfunction with eigenvalue $E_k$ and $\|u_k\|_p=1$, then
$\bar\Lambda_p(u_k,E)=c_k(E-E_k)^2$ with $c_k=\int_\Omega|u_k|^{2(p-1)}dx$; hence
$\Lambda_p(E)\le\inf_kc_k(E-E_k)^2$, the infimum being taken over all normalized eigenpairs.
\item[(b)] $\Lambda_p(E)=0$ if and only if $E\in\sigma_p$.
\end{enumerate}
\end{proposition}
The proof is given in Appendix~\ref{app.proofs}. Thus the zeros of the loss curve are exactly the
eigenvalues, without any orthogonality; this is the justification of the method for $p\ne2$. For
$p=2$ one has $c_k=1$ and (a) gives $\Lambda(E)\le\min_k(E-E_k)^2$, which together with the lower
bound of Section~\ref{sec.approach} yields the identity \eqref{cota.igual}.

The constants $c_k$ determine the curvature of the parabolas in (a). They can be estimated in terms of
$|\Omega|$, $E_k$ and Sobolev constants. Let $q=2(p-1)$ and, for $p<d$, $p^*=dp/(d-p)$; let $S_*$ be the
constant in $\|v\|_{p^*}\le S_*\|\nabla v\|_p$ and, when $W^{1,p}_0(\Omega)\hookrightarrow L^q(\Omega)$,
$S_q$ the constant in $\|v\|_q\le S_q\|\nabla v\|_p$, $v\in W^{1,p}_0(\Omega)$.

\begin{proposition}[Bounds on the constants $c_k$]\label{prop.ck}
Let $d\ge2$ and let $u_k$ be an eigenfunction of \eqref{p-eigen} with eigenvalue $E_k$ and
$\|u_k\|_p=1$, $c_k=\int_\Omega|u_k|^q\,dx$.
\begin{enumerate}
\item[(a)] If $1<p<2$, then $S_*^{-q(1-\theta)/\theta}E_k^{-\alpha}\le c_k\le|\Omega|^{(2-p)/p}$, where
$\theta\in(0,1)$ is given by $\frac1p=\frac\theta q+\frac{1-\theta}{p^*}$ and $\alpha=\frac{q(1-\theta)}{p\theta}>0$.
\item[(b)] If $p=2$, then $c_k=1$.
\item[(c)] If $p>2$, then $c_k\ge|\Omega|^{(2-p)/p}$. If, in addition, $p\ge d$, or $p<d$ and
$d(p-2)\le2p(p-1)$ (which holds for every $p>2$ when $d\le11$), then
$W^{1,p}_0(\Omega)\hookrightarrow L^q(\Omega)$ and $c_k\le S_q^{\,q}E_k^{q/p}=S_q^{2(p-1)}E_k^{2(p-1)/p}$.
\end{enumerate}
\end{proposition}
The proof is given in Appendix~\ref{app.proofs}. The examples of this section satisfy the hypotheses
of (c): $d=2$, $p=2.2>d$, and $d=3$, $p=2.5$, with $q=3\le p^*=15$. The constant $S_*$ depends only on $d$ and $p$, whereas $S_q$ depends in general on
$\Omega$ (in particular when $p\ge d$). By Proposition~\ref{prop.plap}(a) the loss curve lies under the parabola $c_k(E-E_k)^2$ near every eigenvalue, so
every isolated eigenvalue produces a well of the loss curve. For $1<p<2$ the curvature of these bounding
parabolas is bounded above uniformly in $k$ but may decrease as $k$ grows (the bounding parabolas may
flatten), while for $p>2$ it is bounded below uniformly but the upper bound grows like
$E_k^{2(p-1)/p}$ (the bounding parabolas may steepen).

\begin{remark}[What is lost with respect to the linear case]\label{rem.plap.lost}
(1) The proof of (b) is by compactness and gives no lower bound of the form
$\Lambda_p(E)\ge c\,\dist(E,\sigma_p)^2$: there is no analogue of the Weinstein bound, a small value
of the loss does not certify an eigenvalue, and a minimum with no independent reference is only a
candidate. (2) The constants $c_k$ depend on the eigenfunction, so the curvature of the wells is not
known a priori. (3) For fixed $u$, $E\mapsto\int(\Delta_pu+E|u|^{p-2}u)^2$ is still a parabola,
minimized at
\[
E_{\rm PINN}(u)=-\frac{\int_\Omega\Delta_pu\,|u|^{p-2}u\,dx}{\int_\Omega|u|^{2(p-1)}dx},
\]
which we use as the refined eigenvalue (the frozen-network refinement of
Section~\ref{sec.certificate}). For $p\ne2$ this is not the variational quotient
$\int|\nabla u|^p/\int|u|^p$, so its error is in general of first order in the error of the
eigenfunction (instead of second order for $p=2$).
\end{remark}

\begin{lemma}[Degeneracy at critical points]\label{lem.critical}
Let $p>2$, $u\in C^2$ near $x_0$ and $\nabla u(x_0)=0$. Then $\Delta_pu(x_0)=0$; hence the residual
of \eqref{p-eigen} at $x_0$ equals $E|u(x_0)|^{p-2}u(x_0)$, and an eigenfunction cannot be $C^2$ near
a critical point where it does not vanish.
\end{lemma}
\begin{proof}
For $p>2$, $\Phi(\xi)=|\xi|^{p-2}\xi$ is $C^1$ with $D\Phi(0)=0$, so $\Phi(\nabla u)$ is $C^1$ and
\[
\Delta_pu(x_0)=\operatorname{tr}\big(D\Phi(\nabla u(x_0))\,D^2u(x_0)\big)=0. \qedhere
\]
\end{proof}
Our networks are smooth, so by Lemma~\ref{lem.critical} the pointwise residual of a trained network
cannot vanish at its critical points (e.g.\ at the maximum of a first eigenfunction). This suggests,
although it does not prove it for a finite network and a sampled loss, that the trained loss keeps a
positive floor at the eigenvalues, growing with $p$ and with the size of the critical set; this is
what we observe below.

\subsection{Numerical results}
\emph{Settings.} Network $3\times128$ $\tanh$ times $B(x)$, $N=5000$ fresh points per epoch, $\mu=200$,
Adam with learning rate decaying from $10^{-3}$ to $10^{-5}$ (cosine schedule) over 1500 epochs at
each $E$, no early stopping, warm start along the grid, seed 0 unless stated. Proposition~\ref{prop.penalized}
covers only the linear case; by analogy we kept $\mu$ well above $\frac12\max\Lambda_p$ on the window
(the largest values are $\approx54$ in two dimensions and $\approx90$ in three), the level below which
the linear theory predicts collapse to $u=0$; this choice is heuristic. The coarse grid has step $h=1$
(two dimensions) or $h=2$ (ball); around each detected minimum a local three-point scan with half the
step is performed. The loss values are evaluated on a fresh sample of $4\times10^4$ points and
normalized by $(\int|u|^p)^{2(p-1)/p}$. We report $E_{\rm PINN}$ of Remark~\ref{rem.plap.lost}(3), and
$E_{\rm curve}$, the vertex of the parabola through the lowest node of the coarse curve and its two
neighbours.

\emph{References.} Exact eigenvalues are not known for $p\ne2$, so we compare with independent
reference values and report \emph{relative differences}. (1) The first eigenvalue of a ball is
radial and is computed by shooting for the radial ordinary differential equation; the values are
independent of the integration tolerances to $4\times10^{-10}$ (for $p=2$ this reproduces $j_{0,1}^2$
and $\pi^2$; for $(p,d)=(2.2,2)$ it gives $6.53174$, against $6.5320$ in \cite{HJ11}); higher zeros
give radial eigenvalues with nodal circles. (2) For other eigenvalues we use eigenfunctions with a
prescribed nodal structure: if $v$ is the first eigenfunction of a nodal domain whose odd reflections
tile $\Omega$ (sectors of angle $\pi/m$ in the disk; half square, quarter squares and triangles in
the square; half ball), the reflected function is an eigenfunction of $\Omega$ with the same
eigenvalue. These first eigenvalues are computed with an independent $P_1$ finite element
minimization of the Rayleigh quotient (four meshes, Richardson extrapolation, observed order $2$),
validated against the exact values for $p=2$ and against shooting for $p\neq2$ (agreement to
$10^{-5}$); the extrapolation changes the finest-mesh values by at most $2.5\times10^{-4}$ in
relative terms, which we take as a conservative estimate of their uncertainty. (3) For the first two
eigenvalues of the disk and the square, the values of Hor\'ak \cite{HJ11}. We stress that \emph{the
references (2) are eigenvalues with a known eigenfunction symmetry; we do not claim that they are
the third, fourth or fifth eigenvalues in the variational ordering}, except for the first and second
ones, for which our references reproduce the values $18.395$ and $15.144$ of \cite{HJ11} ($18.3929$
and $15.1429$). The index $k$ in Table~\ref{tab.plap} is the order of detection in the window. None
of the three methods is exact; the evidence is the agreement between independent computations.

\begin{table}[tbp]
\centering\small
\caption{Eigenvalues of the $p$-Laplacian detected as local minima of the loss curve \eqref{loss.p} (seed 0). $E_{\rm ref}$: reference value (type of eigenfunction and method in brackets); $E_{\rm PINN}$: refined value; $E_{\rm curve}$: parabolic vertex on the coarse grid; relative differences with respect to $E_{\rm ref}$ in parentheses (\%); $\bar\Lambda_p^{1/2}$ at the minimum.}
\label{tab.plap}
\begin{tabular}{clcccc}
\toprule
$k$ & eigenfunction [reference] & $E_{\rm ref}$ & $E_{\rm PINN}$ (rel.\ diff.) & $E_{\rm curve}$ (rel.\ diff.) & $\bar\Lambda_p^{1/2}$ \\
\midrule
\multicolumn{6}{l}{Unit disk, $p=2.2$, window $[3,58]$}\\
1 & radial, positive [shooting; 6.5320 in \cite{HJ11}] & 6.5317 & 6.5337 ($+0.03$) & 6.55 ($+0.28$) & 0.34 \\
2 & 1 nodal diameter [FEM; 18.395 in \cite{HJ11}] & 18.3929 & 18.4069 ($+0.08$) & 18.43 ($+0.22$) & 0.91 \\
3 & 2 nodal diameters [FEM sector $\pi/2$] & 34.8082 & 34.8043 ($-0.01$) & 34.85 ($+0.11$) & 0.88 \\
4 & radial, 1 nodal circle [shooting] & 41.7886 & 41.9615 ($+0.41$) & 42.70 ($+2.17$) & 2.18 \\
5 & 3 nodal diameters [FEM sector $\pi/3$] & 56.1099 & 56.1166 ($+0.01$) & 56.17 ($+0.11$) & 1.29 \\
\midrule
\multicolumn{6}{l}{Square $[0,2]^2$, $p=2.2$, window $[3,37]$}\\
1 & positive [FEM; 5.4952 in \cite{HJ11}] & 5.4950 & 5.4902 ($-0.09$) & 5.51 ($+0.28$) & 0.47 \\
2 & diagonal nodal line [FEM; 15.144 in \cite{HJ11}] & 15.1429 & 15.1440 ($+0.01$) & 15.15 ($+0.02$) & 0.44 \\
3 & $2\times2$ cells [$2^{p}E_1$] & 25.2487 & 25.2285 ($-0.08$) & 25.23 ($-0.07$) & 1.01 \\
4 & both diagonals nodal [scaled FEM] & 32.4595 & 32.4615 ($+0.01$) & 32.51 ($+0.15$) & 0.97 \\
\midrule
\multicolumn{6}{l}{Unit ball of $\R^3$, $p=2.5$, window $[8,46]$}\\
1 & radial, positive [shooting] & 14.1112 & 14.1004 ($-0.08$) & 14.13 ($+0.15$) & 0.63 \\
2 & 1 nodal plane [FEM half ball] & 35.6315 & 35.6283 ($-0.01$) & 35.64 ($+0.02$) & 1.85 \\
\bottomrule
\end{tabular}
\end{table}

\begin{figure}[tbp]
\centering
\includegraphics[width=.325\textwidth]{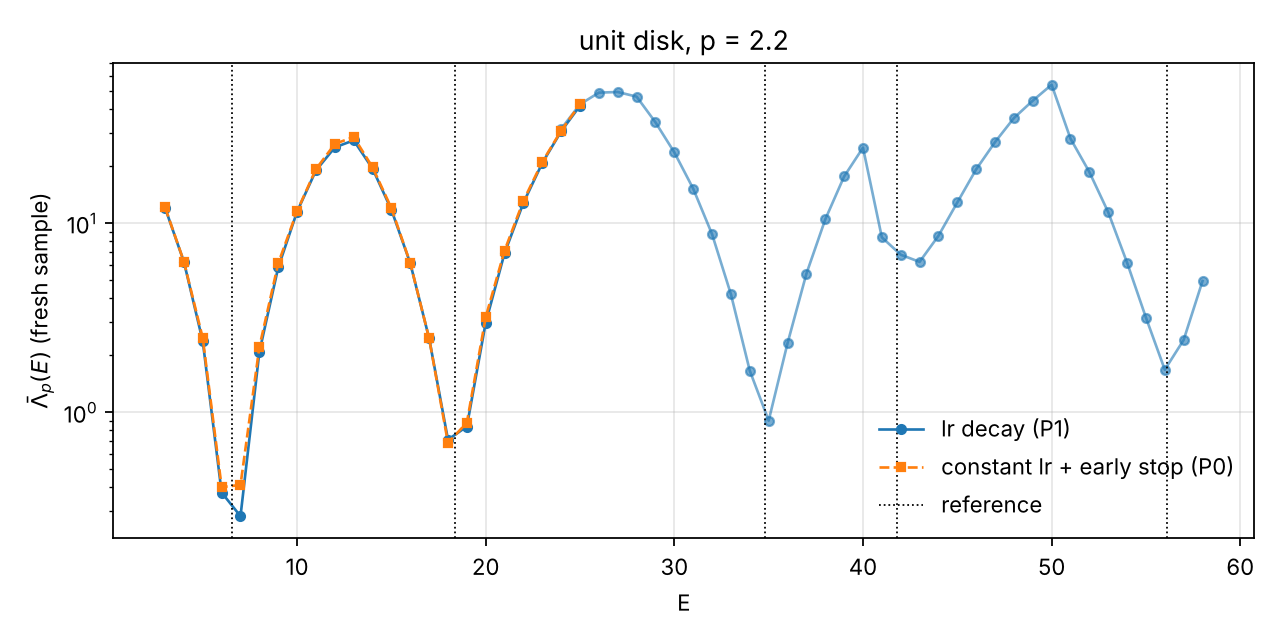}
\includegraphics[width=.325\textwidth]{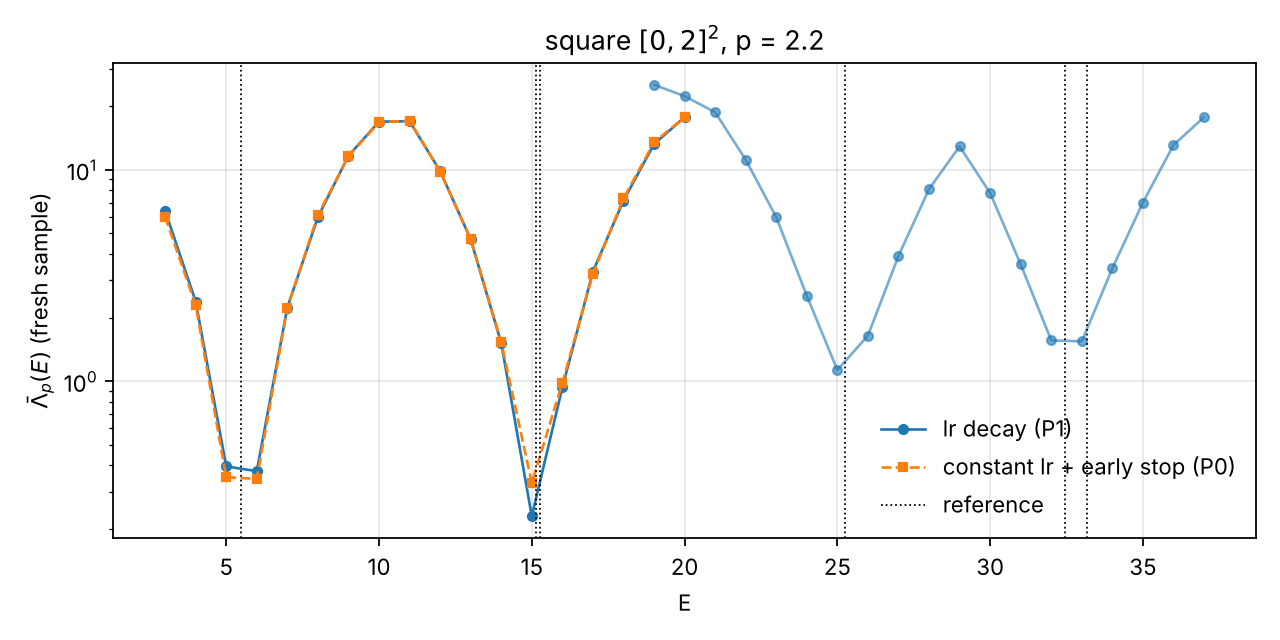}
\includegraphics[width=.325\textwidth]{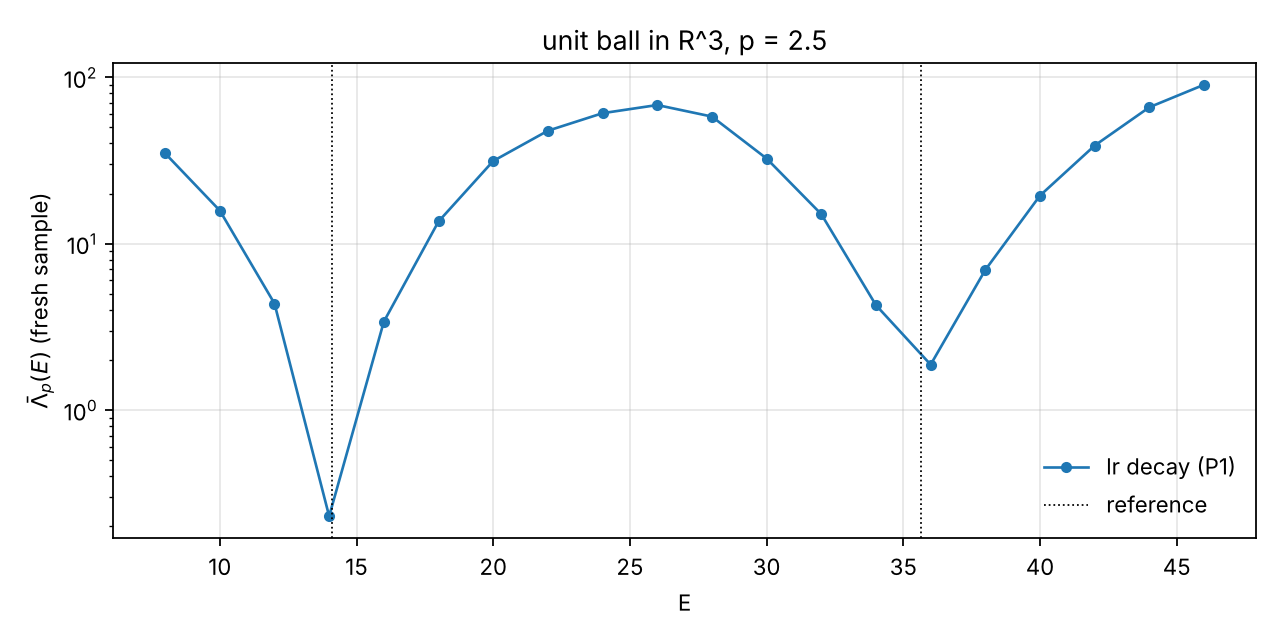}
\caption{Loss curves $\bar\Lambda_p(E)$ (fresh sample, logarithmic scale): unit disk, $p=2.2$ (left), square $[0,2]^2$, $p=2.2$ (centre), unit ball of $\R^3$, $p=2.5$ (right). Dotted lines: reference eigenvalues. P1: learning-rate decay; P0: constant learning rate with early stopping.}
\label{fig.plap}
\end{figure}

Table~\ref{tab.plap} and Fig.~\ref{fig.plap} show that the first and second eigenvalues agree with
the references to within $0.1\%$ in the three domains. Beyond them, the same scan detects further
eigenvalues: in the disk, the third to fifth minima match the eigenvalues with two nodal diameters,
one nodal circle and three nodal diameters, and in the square the third and fourth minima match the
$2\times2$ cell mode and the mode with both diagonals as nodal lines, with relative differences
below $0.1\%$ except for the radial mode with a nodal circle ($+0.41\%$; its gradient vanishes on a
whole circle, which gives the largest residual floor). For the ball of $\R^3$ the shooting value
$14.1112$ and the axisymmetric finite element value $35.6315$ are matched to $0.08\%$ and $0.01\%$.
The rows $k\ge3$ and the ball of $\R^3$ are single runs (seed 0).
Table~\ref{tab.plap.seeds} repeats the first two eigenvalues of the disk and of the square with
three seeds (local three-node windows, $h=0.5$): the spread of $E_{\rm PINN}$ is $0.004$--$0.018$ and the
largest relative difference $0.15\%$, whereas the vertex $E_{\rm curve}$ is much less reproducible
(from $+0.49\%$ to $-1.64\%$ for the first eigenvalue of the square).

\begin{table}[tbp]
\centering\small
\caption{Seed dependence for $p=2.2$ (seeds 0, 1, 2; local three-node windows, $h=0.5$, settings of Table~\ref{tab.plap}). For the second eigenvalue of the square the close eigenvalues $15.1429$ (diagonal nodal line) and $15.2808$ (midline) lie in the same minimum of the loss curve; each seed converges to one of them, no mean is given, and the difference is measured to the nearer one.}
\label{tab.plap.seeds}
\setlength{\tabcolsep}{4pt}
\begin{tabular}{lcccccc}
\toprule
eigenvalue & $E_{\rm ref}$ & $E_{\rm PINN}$ (seeds 0, 1, 2) & mean & spread & max.\ rel.\ diff.\ (\%) & $E_{\rm curve}$ (seeds 0, 1, 2)\\
\midrule
disk $E_1$ & 6.5317 & 6.5320, 6.5329, 6.5285 & 6.5311 & 0.0044 & 0.05 & 6.54, 6.54, 6.54\\
disk $E_2$ & 18.3929 & 18.4069, 18.3890, 18.3968 & 18.3976 & 0.0179 & 0.08 & 18.41, 18.45, 18.46\\
square $E_1$ & 5.4950 & 5.4902, 5.4867, 5.4966 & 5.4912 & 0.0100 & 0.15 & 5.52, 5.49, 5.41\\
square $E_2$ & 15.1429 / 15.2808 & 15.1440, 15.2467, 15.1396 & -- & -- & 0.22 & 15.16, 15.51, 15.42\\
\bottomrule
\end{tabular}
\end{table}

Two caveats apply. First, close eigenvalues are not separated: in the square, $15.1429$ (diagonal
nodal line) and $15.2808$ (midline nodal line, the other mode into which the double eigenvalue
$5\pi^2/4$ of the Laplacian splits at $p=2.2$) give a single minimum, and which of the two the network
converges to depends on the seed (seeds 0 and 2 gave the diagonal mode, seed 1 the midline mode, Table~\ref{tab.plap.seeds}); the same
happens for $32.46$ and $33.17$ (three strips). Second, the reference for $35.63$ is an eigenvalue of
the ball with one nodal plane and an upper bound for the second eigenvalue; that it is the second
eigenvalue is not proved.

\subsection{Error sources}
Tables~\ref{tab.plap.conv} and \ref{tab.plap.conv2} in Appendix~\ref{app.conv} vary one setting at a
time for the first and second eigenvalues of the disk, $p=2.2$ (single seed except where stated). We identify the following error
sources.
\begin{enumerate}
\item \emph{Optimization noise.} With a constant learning rate and early stopping (protocol P0, as in
the linear examples) the loss at a fixed $E$ fluctuates between seeds by $\pm0.1$--$0.2$, while the
curve rises only by $\approx0.27$ at distance $0.5$ from the eigenvalue; the vertex of the curve then
has errors of up to $2\%$ of either sign even with the loss \eqref{loss.p} ($-0.17$, $+0.47$, $+2.26\%$
for three seeds). A decaying learning rate reduces the
fluctuation to $\approx0.01$--$0.02$ and the vertex difference to $0.2\%$.
\item \emph{Residual floor near critical points.} After the decay the loss at the eigenvalue is
$\approx0.13$, and $50$--$60\%$ of it comes from the $2\%$ of the sample points with the smallest
$|\nabla u|$, consistent with Lemma~\ref{lem.critical}. The floor decreases slowly with capacity and
training ($0.150$ for $2\times32$, $0.096$ for $4\times256$; $0.049$ with 5000 epochs), and since it
is not symmetric in $E$ it biases the minimum of the curve by $+0.2$ to $+0.8\%$; this bias is not
removed by a larger network, more points or longer training.
\item \emph{Grid in $E$.} Refining the grid makes the curve estimate \emph{worse}
($+0.18\%$ and $+0.78\%$ for $h=0.5$ and $0.25$, one seed each; for $h=0.125$ the three loss values
are monotone and the stencil no longer brackets the minimum), because the differences of the loss
between neighbouring nodes ($\approx c_1h^2$) fall below the fluctuations of the floor. The step should
satisfy $c_1h^2\gg$ (fluctuation), here $h\gtrsim0.5$; a coarse grid followed by the frozen-network
refinement is cheaper and more accurate.
\item \emph{Network capacity and Monte Carlo size} are minor at fixed protocol: $E_{\rm PINN}$ changes
by at most $0.05\%$ from $2\times32$ to $4\times256$ neurons, and by $0.13\%$ for $N=2000$ and $0.00\%$
for $N=20000$.
\end{enumerate}
The refined value $E_{\rm PINN}$ is insensitive to all these settings (relative difference at most
$0.05\%$ for every configuration with learning-rate decay, $\ge1500$ epochs and $N\ge5000$), because it
is computed from the trained function alone. The same holds for the second eigenvalue of the disk
(Table~\ref{tab.plap.conv2}): over three seeds, constant learning rate, a $2\times64$ network and 5000
epochs, $E_{\rm PINN}$ at the node $E=18.5$ stayed within $0.08\%$ of $18.3929$ (from $-0.14\%$ to
$+0.08\%$ over the three nodes), whereas $E_{\rm curve}$ ranged from $+0.11\%$ to $+1.68\%$ and did not
improve with longer training; the floor ($\approx0.85$) is about seven times that of the first
eigenvalue.
Hence the loss curve should be used to \emph{detect} the eigenvalues and the frozen-network quotient
to \emph{compute} them.

\subsection{Range of \texorpdfstring{$p$}{p}}
Table~\ref{tab.plap.p} shows the first eigenvalue of the disk for several $p$, with three seeds and
the dependence of $E_{\rm PINN}$ on the node at which the network was trained. The residual floor of
Lemma~\ref{lem.critical} grows with $p$ (the eigenfunction is only $C^{1,1/(p-1)}$ at its maximum),
and so does this node dependence: at the grid minimum the largest relative difference is $0.05\%$,
$0.31\%$ and $1.45\%$ for $p=2.2$, $3$, $4$, but over the three central nodes $E_{\rm PINN}$ ranges from
$-0.09\%$ to $+0.10\%$, from $-0.77\%$ to $+0.81\%$ and from $-2.56\%$ to $+1.48\%$, respectively. At
$p=4$ the curve is flat ($\bar\Lambda_p=4.4$--$8.0$ between $E=14$ and $16$) and 2 of the 15 trainings
collapsed to $u\equiv0$. The method is therefore validated for $p\in[2.2,3]$ (and $p=2.5$ in three
dimensions); $p=4$ is marginal.

\begin{table}[tbp]
\centering\small
\caption{First eigenvalue of the $p$-Laplacian in the unit disk (reference by radial shooting; settings of Table~\ref{tab.plap}; steps $h=0.5$, $0.75$, $1$ for $p=2.2,3,4$). $E_{\rm PINN}$ at the grid minimum for seeds 0, 1, 2, and range of $E_{\rm PINN}$ over the three central nodes and all seeds. $p=1.5$: unregularized loss, seed 0, no minimum near $E_1$.}
\label{tab.plap.p}
\begin{tabular}{ccccc}
\toprule
$p$ & $E_1$ & $E_{\rm PINN}$ (seeds 0, 1, 2) & max.\ rel.\ diff.\ (\%) & range over nodes (rel.\ diff.\ \%)\\
\midrule
1.5 & 4.0178 & 5.2887 & $+31.6$ & 5.29--7.84\\
2.2 & 6.5317 & 6.5320, 6.5329, 6.5285 & 0.05 & 6.526--6.538 ($-0.09$ to $+0.10$)\\
3 & 9.8315 & 9.801, 9.827, 9.814 & 0.31 & 9.76--9.91 ($-0.77$ to $+0.81$)\\
4 & 14.6816 & 14.895, 14.613, 14.533 & 1.45 & 14.31--14.90 ($-2.56$ to $+1.48$)\\
\bottomrule
\end{tabular}
\end{table}

\emph{The case $p<2$ and regularization.} For $p=1.5$ the loss \eqref{loss.p} fails: the curve has
no minimum near $E_1=4.0178$ ($\bar\Lambda_p=24$--$58$). For $p<2$ the strong form $\Delta_pu$ of a
smooth network is singular at its critical points (like $|x-x_0|^{p-2}$), and the Monte Carlo
estimator of the loss has infinite variance for $p\le3/2$ in two dimensions: the failure is a
consequence of the regularity required by the strong form. Replacing $\Delta_pu$ in the loss by the
regularized operator $\operatorname{div}((|\nabla u|^2+\nu^2)^{(p-2)/2}\nabla u)$, $\nu>0$, removes the
singularity, and in our test it cures the failure (Table~\ref{tab.plap.reg}; unit disk, first
eigenvalue, seed 0): the curve has a clean minimum near $E=4$ with $\bar\Lambda_p^{1/2}\approx0.3$, as
for $p=2.2$. The regularized problem is not homogeneous, so its eigenvalue $E_\nu$ depends on the
normalization; with $\int|u|^p=1$, radial shooting gives $E_\nu=3.9572$, $4.0099$, $4.0170$ for
$\nu=0.3$, $0.1$, $0.03$. The PINN values differ from $E_\nu$ by $-0.23$, $-0.28$, $-0.29\%$ and from
$E_1$ by $-1.74$, $-0.47$, $-0.31\%$; for $\nu\le0.1$ the total difference is dominated by the PINN error
(whose origin we have not identified) rather than by the regularization shift ($\le0.2\%$). For $p=4$ the
regularization ($\nu=0.1$) brings no improvement, as expected: for $p>2$ the difficulty is the defect at
the critical points (Lemma~\ref{lem.critical}), not a singularity of the operator.

\begin{table}[tbp]
\centering\small
\caption{Regularized operator $\operatorname{div}((|\nabla u|^2+\nu^2)^{(p-2)/2}\nabla u)$ in the loss, unit disk, first eigenvalue, seed 0 ($E_1=4.0178$ for $p=1.5$, $14.6816$ for $p=4$). $E_\nu$: radial shooting for the regularized problem with $\int|u|^p=1$. Row $\nu=0$: unregularized loss.}
\label{tab.plap.reg}
\begin{tabular}{ccccccc}
\toprule
$p$ & $\nu$ & $E_\nu$ & $E_{\rm PINN}$ & rel.\ diff.\ vs $E_\nu$ (\%) & rel.\ diff.\ vs $E_1$ (\%) & $\bar\Lambda_p^{1/2}$\\
\midrule
1.5 & 0 & -- & 5.29 (no minimum) & -- & $+31.6$ & 4.92\\
1.5 & 0.3 & 3.9572 & 3.9480 & $-0.23$ & $-1.74$ & 0.31\\
1.5 & 0.1 & 4.0099 & 3.9988 & $-0.28$ & $-0.47$ & 0.32\\
1.5 & 0.03 & 4.0170 & 4.0054 & $-0.29$ & $-0.31$ & 0.35\\
4 & 0.1 & 14.7490 & 14.9459 & $+1.33$ & $+1.80$ & 2.47\\
\bottomrule
\end{tabular}
\end{table}
\end{revblock}

\section{Analysis of the practical method: stability, choice of \texorpdfstring{$\mu$}{mu}, and cost}\label{sec.analysis}

\subsection{Stability of the detection and spurious minima}\label{sec.stability}
\begin{revblock}
The identity \eqref{cota.igual} concerns exact minimizers, whereas the network class $\mathcal N$ of
functions $u=B\,\mathrm{NN}_\theta$ is finite dimensional, training stops after finitely many steps,
and the integrals are sampled. Let $\tilde\Lambda_i$ be the computed (normalized) loss at the node
$E^i=E_*+(i-1)h$. The error splits as
\[
\tilde\Lambda_i-\Lambda(E^i)=
\underbrace{\Lambda(u^i,E^i)-\inf_{\mathcal N}\Lambda(\cdot,E^i)}_{\text{optimization}\ \ge0}
+\underbrace{\inf_{\mathcal N}\Lambda(\cdot,E^i)-\Lambda(E^i)}_{\text{approximation}\ \ge0}
+\underbrace{\tilde\Lambda_i-\Lambda(u^i,E^i)}_{\text{sampling}} ;
\]
since $\mathcal N\subset \mathrm{Dom}(L)$, approximation and optimization errors can only \emph{raise} the loss
curve. We assume only a bound on the total error,
\begin{equation}\label{hyp.delta}
|\tilde\Lambda_i-\ell(E^i)^2|\le\delta,\qquad i=1,\dots,J,
\end{equation}
and call an interior node $E^m$ a discrete local minimum if
$\tilde\Lambda_m\le\min\{\tilde\Lambda_{m-1},\tilde\Lambda_{m+1}\}$.

\begin{proposition}[Stability of the loss-curve characterization]\label{prop.stability}
Assume \eqref{hyp.delta}, let $\epsilon>0$ be the detection threshold and
$\tau=\sqrt{h^2/4+2\delta}$.
\begin{enumerate}
\item[(a)] (No detections away from the spectrum.) If $\tilde\Lambda_i\le\epsilon$, then
$\ell(E^i)\le\sqrt{\epsilon+\delta}$.
\item[(b)] (No missed eigenvalues by thresholding.) Every eigenvalue $E_k\in[E_*,E^*]$ has a node
with $|E^i-E_k|\le h/2$ and $\tilde\Lambda_i\le h^2/4+\delta$; hence all of them are detected if
$\epsilon\ge h^2/4+\delta$.
\item[(c)] (Local minima are close to the spectrum.) Every discrete local minimum satisfies
$\ell(E^m)\le\max\{h,\ h/2+\delta/h\}$, irrespective of its value.
\item[(d)] (Isolated eigenvalues produce local minima.) If
$[E_k-\tau-h,E_k+\tau+h]\subset[E_*,E^*]$ and $\dist(E_k,\sigma\setminus\{E_k\})>2\tau+h$, there is a
discrete local minimum $E^m$ with $|E^m-E_k|\le\tau$.
\end{enumerate}
\end{proposition}
The proof is given in Appendix~\ref{app.proofs}. Since $\tau\ge\sqrt{2\delta}$, part (d) can only
apply to eigenvalues whose distance $\gamma$ to the rest of the spectrum satisfies roughly
$\delta\ll\gamma^2/8$; the proposition is informative only in this regime, which we check below.

\begin{remark}[Interpretation]\label{rem.stability}
(1) Imperfect training cannot create detections far from the spectrum: every node below the
threshold, and every discrete local minimum, lies within an explicit distance of a true
eigenvalue, provided \eqref{hyp.delta} holds. What it can do is produce several minima around the
same eigenvalue, or raise the curve above the threshold so that an eigenvalue is missed. Refining the
grid below $h\simeq\sqrt{2\delta}$ does not improve the localization (the bound in (c) is minimal
there). (2) The loss curve only sees the set $\sigma$: a multiple eigenvalue produces a single
minimum and the network converges to some element of the eigenspace. Multiplicities would have to
be recovered separately, e.g.\ by training at the detected value from several initializations and
computing the rank of the Gram matrix of the resulting functions; we have not implemented this.
(3) Two eigenvalues are separated by a bump of height $(E_{k+1}-E_k)^2/4$; if
$E_{k+1}-E_k\le2\sqrt{2\delta}$ the bump is below the error level and the number of eigenvalues in
such a cluster cannot be determined from the curve: the method returns the cluster as a single
approximate eigenvalue (Section~\ref{sec.fem}). (4) Proposition \ref{prop.stability} is conditional
on \eqref{hyp.delta}; we do not prove a bound on $\delta$ in terms of the network size, the number of
epochs and of sampling points. The approximation error tends to zero as the width grows (networks
with smooth non-polynomial activation are dense in $C^2(\overline\Omega)$
\cite{Hornik1991,Pinkus1999}; rates for $\tanh$ networks are given in
\cite{DeRyck-Lanthaler-Mishra}), but no convergence guarantee for the optimizer is available, and the
spectral bias of gradient-based training \cite{Rahaman-etal} suggests that $\delta$ grows with $E$,
in agreement with the residuals in our tables. The error can, however, be \emph{measured} a
posteriori, as we do next.
\end{remark}

\emph{Experiments.} (i) \emph{Seeds and certificates.} We scanned the window $[15,55]$ of the unit
square (step 5, $\mu=450$) with seeds 0, 1, 2. All three runs gave local minima at $E=20$ and $E=50$
only, i.e.\ no spurious minimum, and refined values $19.7394$, $19.7397$, $19.7450$ for $2\pi^2$ and
$49.3486$, $49.3485$, $49.3474$ for $5\pi^2$. Over the 14 minima detected in our unit-square runs
(these scans, the scan of Table~\ref{table.1} and the cold-start runs) the Weinstein radius $r$
exceeded the actual error by a factor $120$--$4300$, and the Kato--Temple bound $r^2/\gamma$ by a factor
$2.7$--$66$ (median $7.5$): both bounds held, and the Kato--Temple bound is the informative one.

(ii) \emph{Size of $\delta$.} We measured $\delta_i=|r_i^2-\ell(E^i)^2|$ at every node, with $r_i^2$ the
normalized loss on a fresh sample. In the three seeds of the offset scan of Section~\ref{sec.cost}
(18 nodes on $[17.3,102.3]$) the maxima were $3.5$, $2.9$, $2.6$; in the three seeds on $[15,55]$ they
were $9.8$, $7.8$, $12.1$, attained at the mid-gap node $E=35$, where $\ell^2\approx206$ and a relative
error of a few percent in $r^2$ gives $\delta\approx10$ (medians $0.15$--$0.25$). In the 80-node scan of
Table~\ref{table.1} the maximum over all nodes was $1749$, but after excluding $9$ nodes (5 collapsed,
4 branch-switched or lagging nodes: after the re-initialization at $E\approx66.5$ the network sat on the
$5\pi^2$ branch, and already at $E=65.1$, past the mid-gap $64.15$, the warm-started network had
lagged on that branch) it was $2.35$
(median $0.26$). On the remaining (clean) nodes $\delta\le12.1$, well below $\gamma^2/8=48.7$ and $109.6$
for the gaps $\gamma=19.7$ and $29.6$ of the unit square, and the hypotheses of (d) hold (e.g.\
$h=5$, $\delta=12.1$ give $2\tau+h=16.0<19.7$). They do not make the bound (c) sharp: $\delta$ is
comparable to $h^2/4$ at the mid-gap nodes, where it is dominated by the training floor at large
$\ell$. Large values of $\delta$ occurred only at collapsed nodes and at branch-switched or lagging nodes, where
the network sits on the eigenfunction of an eigenvalue different from the nearest one; the
eigenvalues missed in the annulus and for $\omega=10$ (Section~\ref{sec.fem}) are instances of the
second kind. Both are visible in the network itself: the first by a vanishing mass, the second by a
jump of the Rayleigh quotient along the scan.

(iii) \emph{Collapse.} Spurious zeros of the \emph{recorded} curve were observed, but
only in one form: the collapse to the trivial state $u\equiv0$. In the scan of Table~\ref{table.1}
the network was re-initialized at $E=99.7$, collapsed to $u\equiv0$ and stayed there up to $E=105$
(Fig.~\ref{fig.collapse}). The recorded residual term dropped to $\approx10^{-5}$, the lowest value of
the whole curve, which any threshold would accept; but the mean of $u^2$ was $\approx10^{-8}$ and the
certificate radius $r$ was $30$--$42$. Collapsed nodes are therefore detected by $\overline{u^2}\approx0$
and a large $r$, and are invisible in the residual term alone. The same happens when $\mu$ is too
small (Section~\ref{sec.mu}). We therefore record, at each node, the normalized loss
$\Lambda(u,E)=r^2$, the mass $\overline{u^2}$ and the Rayleigh quotient, and discard nodes with
$\overline{u^2}<\frac12$.

\begin{figure}[tbp]
\centering
\includegraphics[width=.6\textwidth]{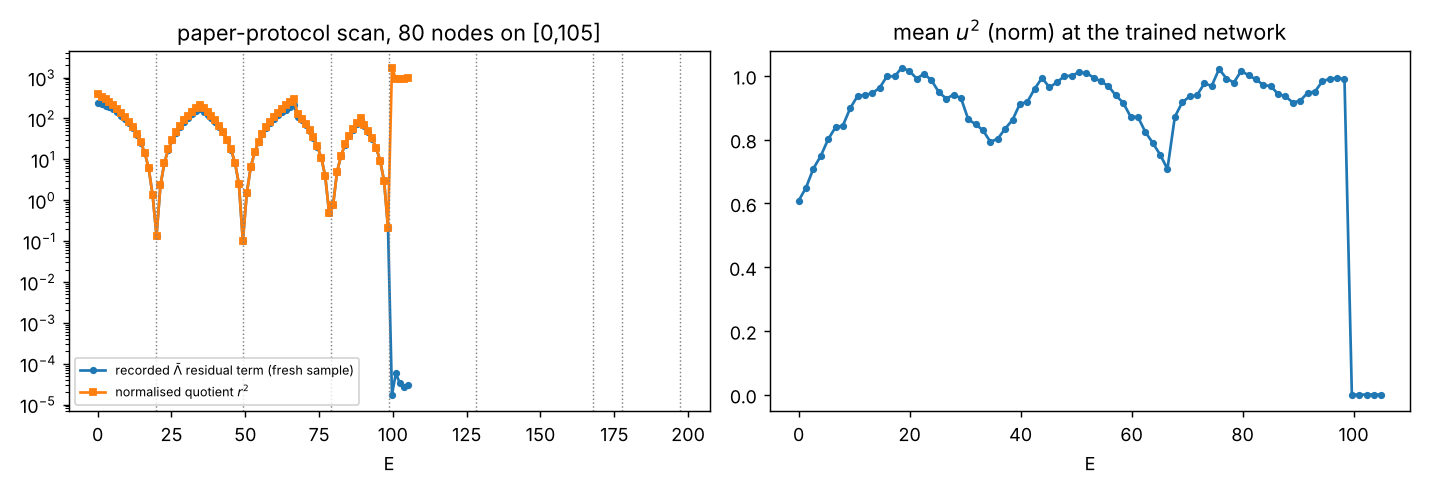}
\caption{The scan of Table~\ref{table.1} (80 nodes on $[0,105]$, $\mu=500$). Left: recorded residual term and normalized loss $r^2$ (dotted lines: eigenvalues). Right: mean of $u^2$. After the re-initialization at $E=99.7$ the network collapses to $u\equiv0$: the recorded residual drops to $10^{-5}$ while $r^2$ jumps to $\approx10^3$.}
\label{fig.collapse}
\end{figure}
\end{revblock}

\subsection{Choice of the penalty parameter}\label{sec.mu}
\begin{revblock}
Proposition~\ref{prop.penalized} gives a rule for $\mu$. Let $[E_*,E^*]$ be the window and
\[
D=\max_{E\in[E_*,E^*]}\ell(E),
\]
i.e.\ half of the largest gap between consecutive eigenvalues meeting the window, or the distance from
an endpoint of the window to the spectrum if this is larger. Let $\mathcal R_\mu(E)$ be the residual term of
\eqref{lossmu} at the minimizer, which is the quantity recorded during the scan.

\begin{corollary}\label{cor.mu}
(i) For every $\mu>0$ the zeros of $\bar\Lambda(E)$ are exactly the eigenvalues, and
$0\le \ell^2-\bar\Lambda(E)\le \ell^4/(4\mu)$ when $\ell^2\le2\mu$. (ii) If $\mu\ge D^2$, then
$\|u_E\|_2^2\ge\frac12$ and $\mathcal R_\mu(E)=\ell^2(1-\ell^2/(2\mu))$ is an increasing function of $\ell=\ell(E)$ on the
window: its zeros and local minima are exactly the eigenvalues. (iii) If $\frac12D^2\le\mu<D^2$,
$\mathcal R_\mu$ has spurious local minima where $\ell$ has a local maximum larger than $\sqrt\mu$. (iv) If
$\mu<\frac12D^2$, the minimizer is $u_E=0$ wherever $\ell(E)^2\ge2\mu$, and the recorded curve has
spurious zeros there.
\end{corollary}
This follows directly from Proposition~\ref{prop.penalized}, since $\mathcal R_\mu=\psi(\ell^2)$ with
$\psi(t)=t(1-t/(2\mu))_+$, which increases on $[0,\mu]$, decreases on $[\mu,2\mu]$ and vanishes for
$t\ge2\mu$. The location of the minima does not depend on $\mu$, and the normalized quotient
$\Lambda(u_E,E)=\ell(E)^2$ is independent of $\mu$.

\emph{Guideline.} (1) Along a warm-started scan it suffices to take $\mu\ge D^2$; we recommend
$\mu\approx2D^2$. $D$ can be estimated from Weyl's law or from a coarse preliminary scan. Since the
gaps do not grow with $E$ in dimension $d\ge2$, $\mu$ need not grow with $E$. (2) A very large $\mu$ is
harmful: the sampled penalty adds a noise floor $\mu\operatorname{Var}(u^2)/N$ to the loss and makes
the problem stiff. (3) Training from scratch at a high level $E$ (cold start, or periodic
re-initialization) is a different matter, discussed below. The settings of all our examples are
compared with the rule in Table~\ref{table.settings}: it holds in all of them except the ball of $\R^4$
(regime (iii), $\mu=200<D^2=216$) and the harmonic potential with $\omega=10$ (regime (iv), collapse at
$E=0$, Fig.~\ref{pot1.fig}); in the triangle, for instance, the window $[0,180]$ starts far below
$E_1=5\pi^2$, so $D=5\pi^2$ and $D^2\approx2435<\mu=15000$.

\emph{Test of the rule.} We used a reduced network ($3\times64$, 3000 points, at most 2000 epochs;
at $\mu=1000$ it gave the same normalized loss values and Rayleigh quotients as the full network to
within the training noise), warm start and two seeds. On the window $[15,55]$ of the unit square,
which contains $2\pi^2$ and $5\pi^2$ (double), $D=14.8$ at the mid-gap $34.5$ and $D^2=219.2$, we used an
offset grid $E=15.6,18.1,\dots,53.1$ (no node on an eigenvalue) and six values of $\mu$
(Table~\ref{table.mu}). The predicted effect is observed: for $\mu<D^2$ the recorded residual has a
spurious interior minimum in the gap ($E\approx33$--$36$, next to the mid-gap) in all six runs, and for
$\mu\ge D^2$ in none of the six runs; for $\mu<D^2/2$ the network collapses exactly at the predicted
nodes. The eigenvalues themselves are located correctly for every $\mu$ (Rayleigh quotient within
$2\times10^{-3}$ of $2\pi^2$ and $10^{-3}$ of $5\pi^2$), and the normalized loss $r^2$ has no spurious
minimum for $\mu\ge75$: the rule concerns the recorded residual that is thresholded, not the
eigenvalue estimate. On the window $[40,60]$, which contains only $5\pi^2$ ($D^2=113.5$, step 2, nine
values of $\mu$ from $1$ to $10^4$), the recorded curves follow the closed form of
Proposition~\ref{prop.penalized} for $60\le\mu\le1000$ (Fig.~\ref{fig.mu}; the mean of $u^2$ follows
$1-\ell^2/(2\mu)$ with an RMS deviation $\approx0.01$ for $\mu\ge100$); the networks collapse at the
far nodes for $\mu<D^2/2$ (at all nodes for $\mu=1,10$, more than predicted, because a warm-started
$u\approx0$ is not left, and at $\mu=30$ the collapsed state carried along by the warm start produces
a spurious zero at $E=42$); the relative error of the eigenvalue is at most $2.9\times10^{-5}$ for
$60\le\mu\le1000$ and $2.0\times10^{-4}$ for $\mu=10^4$, where the curve is distorted (stiff penalty).
A single eigenvalue gives a single basin, so this window cannot distinguish $\mu=60$ from $\mu=113$.

\begin{table}[tbp]
\centering\small
\caption{Test of the rule $\mu\ge D^2$: window $[15,55]$ of the unit square ($D^2=219.2$; the value $\mu=219$ is $\approx D^2$), offset grid of 16 nodes, two seeds, reduced network. Spurious minima: interior local minima of the recorded residual more than $2.5$ away from an eigenvalue, seed 0 / seed 1. Collapsed: nodes with $\overline{u^2}<0.1$, observed / predicted by Proposition~\ref{prop.penalized}. $\hat E$: Rayleigh quotient near $5\pi^2=49.3480$ (seeds 0, 1).}
\label{table.mu}
\begin{tabular}{rrcccc}
\toprule
$\mu$ & $\mu/D^2$ & regime & spurious minima & collapsed & $\hat E$\\
\midrule
30 & 0.14 & (iv) & 3 / 1 ($E=28$--$41$) & 6/6 & 49.3482, 49.3477\\
75 & 0.34 & (iv) & 1 / 1 ($E=33$--$36$) & 2/2 & 49.3487, 49.3476\\
150 & 0.68 & (iii) & 1 / 1 ($E=35.6$) & 0/0 & 49.3473, 49.3478\\
219 & $\approx1.0$ & (ii) & none & 0/0 & 49.3487, 49.3483\\
300 & 1.4 & (ii) & none & 0/0 & 49.3479, 49.3481\\
1000 & 4.6 & (ii) & none & 0/0 & 49.3475, 49.3475\\
\bottomrule
\end{tabular}
\end{table}

\begin{figure}[tbp]
\centering
\includegraphics[width=.45\textwidth]{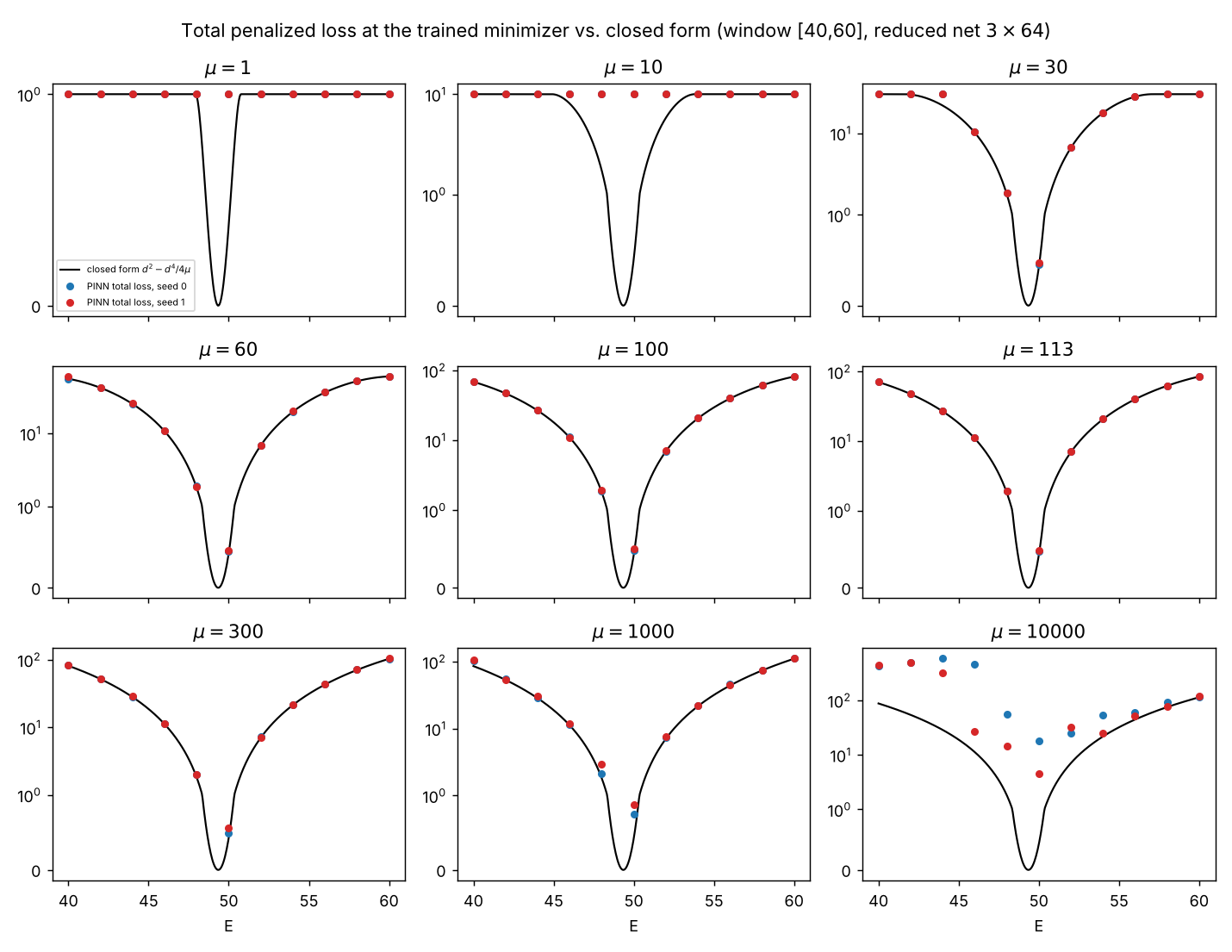}
\caption{Total penalized loss at the trained networks (dots, two seeds) and the closed form $\ell^2-\ell^4/(4\mu)$ of Proposition~\ref{prop.penalized} (solid line), window $[40,60]$ of the unit square, for nine values of $\mu$ between $1$ and $10^4$.}
\label{fig.mu}
\end{figure}

\emph{Cold start.} When the network is trained from scratch at a level $E$, the residual of the
initial network is of order $E^2$ and the amplitude is driven towards $u=0$ unless, heuristically,
$\mu\gtrsim\frac12E^2$. On the window $[95,102]$ (five nodes, fresh network at the first node, then
warm start, at most 5000 epochs) the eigenvalue $10\pi^2$ was reached as follows: $\mu=450$: collapse at
all nodes (one seed); $\mu=3000$: only at the last node, the first four nodes converged to the
eigenfunction of $5\pi^2$ (one seed; similar failures at $8\pi^2$ and $10\pi^2$ in two further cold
chains); $\mu=5000\approx E^2/2$: at 8 of 15 nodes (three seeds; the other nodes on the $5\pi^2$
branch); $\mu=20000$: at none of 15 nodes (all stuck at $2\pi^2$). Thus $\mu\gtrsim E^2/2$ avoids
collapse but is not sufficient, and cold start at a high level is unreliable: the network tends to
fall into the eigenfunction of a lower eigenvalue and only sometimes leaves it. The re-initializations
of the scan of Table~\ref{table.1} ($\mu=500\ll E^2/2$) show the same: branch switch at $E\approx66$ and
collapse at $E\ge99.7$. In the local windows of Tables~\ref{table.2}--\ref{table.4} the cold first node
converged in the ball of $\R^3$ but failed in every window of the triangle. We therefore recommend
warm-start continuation from a window that starts at or below the target, or several independent
restarts, and accepting a detection only through the mass, the Rayleigh quotient and the certificate.
\end{revblock}

\subsection{Computational cost}\label{sec.cost}
\begin{revblock}
\emph{Cost model.} Let $W$ be the number of network parameters, $N$ the number of sampling points and
$d$ the dimension. The loss \eqref{lossmu} needs $\Delta u_\theta$, i.e.\ one backward pass for
$\nabla u_\theta$ and $d$ further passes for the diagonal of the Hessian, so one epoch costs
$c_{\rm ep}=\kappa(d+2)NW$, linear in $d$. If $n_i$ epochs are used at node $E^i$, the scan of a
window with $J$ nodes costs
\begin{equation}\label{cost.scan}
C_{\rm scan}=c_{\rm ep}\sum_{i=1}^Jn_i\le\kappa(d+2)NWJN_{\rm ep}.
\end{equation}
By Proposition~\ref{prop.stability} the grid only needs to separate the eigenvalues, $h\lesssim
\gamma_{\min}/2$, since the accuracy comes from the Rayleigh quotient and not from $h$; hence
$J\approx2(E^*-E_*)/\gamma_{\min}$. With warm start most nodes need few epochs (in the scan of
Table~\ref{table.1}, $302$--$3060$ epochs per node, median $\approx500$, against a maximum of $5000$). The
trainings in different sub-windows, each with its own warm start, are independent, so the
wall-clock time can be divided by the number of processors. A Rayleigh-quotient
PINN with deflation \cite{Rowan-Evans-Maute-Dootsan,Holliday-Lindner-Ditto,E-Yu} computes the
$k$-th eigenpair by minimizing $R(u)$ plus penalties for the overlap with the $k-1$ previous networks;
an epoch needs only first derivatives (cost $\kappa'NW$) plus the evaluation of the previous networks,
so the first $K$ eigenpairs cost $\approx\bar n\,NW(\kappa'K+\frac{\kappa_0}2K^2)$, sequentially, and errors
in the computed eigenfunctions bias all subsequent eigenvalues downwards (Courant--Fischer).

\emph{Measurements.} We compared both approaches on the unit square with the same network, boundary
factor and sample size, on one CPU thread, with three seeds each (Table~\ref{table.cost}). The
deflation baseline minimizes $R(u)$ plus $\beta\sum_{j<k}\langle u,u_j\rangle^2/(\|u\|^2\|u_j\|^2)$
with $\beta=200$, with or without early stopping; the eigenvalue is its Rayleigh quotient. The scan
uses a grid offset from the spectrum ($E=17.3,22.3,\dots,102.3$, $\mu=450$).

\begin{table}[tbp]
\centering\small
\caption{Cost on the unit square (one CPU thread, $3\times128$ network, 5000 points per epoch, shared machine: times indicative). Mean [min, max] over seeds 0, 1, 2 unless stated. The six lowest eigenvalues counted with multiplicity are $2\pi^2$, $5\pi^2$ (double), $8\pi^2$, $10\pi^2$ (double). Refinement times in parentheses.}
\label{table.cost}
\begin{tabular}{lrll}
\toprule
method & wall time (s) & eigenvalues obtained & max.\ rel.\ error\\
\midrule
deflation, $K=6$, early stopping & 247 [229, 262] & 6 (with mult.) & $3.1\times10^{-2}$\\
deflation, $K=6$, no early stopping & 813 [796, 830] & 6 (with mult.) & $2.3\times10^{-3}$\\
scan $[17.3,102.3]$, 18 nodes & 2082 [2023, 2135] (+77--110) & 4 distinct & $4.1\times10^{-5}$\\
scan $[0,105]$, 80 nodes (Table~\ref{table.1}; seed 0) & 6534 (+413) & 4 distinct & $3.5\times10^{-5}$\\
cold start $[95,102]$, 5 nodes, $\mu=5000$ & 818--1008 & $10\pi^2$ at 8/15 nodes & $2.1\times10^{-5}$ (one run)\\
\bottomrule
\end{tabular}
\end{table}

In all three seeds of the offset scan, $2\pi^2$ appears at the boundary node $E=17.3$ ($r^2\approx6.0$,
against $6.6$--$6.8$ at $E=22.3$), not as an interior minimum, so Proposition~\ref{prop.stability}(d)
does not apply to it; it is identified through the Rayleigh quotient at that node ($19.740$). The other
three distinct eigenvalues are interior minima.

For the lowest eigenvalues of a two-dimensional problem the Rayleigh-quotient PINN with deflation is
\emph{cheaper} in raw wall time: the 18-node scan is $2.6$ times slower than deflation without early
stopping and $8.4$ times slower than deflation with early stopping; and deflation \emph{resolves
multiplicities}, which the scan does not. The scan was more accurate in these runs ($4\times10^{-5}$
against $2.3\times10^{-3}$ and $3.1\times10^{-2}$); we did not push the training of the deflation
baseline further, so the cost of deflation at equal accuracy was not measured. We also did not test
the regime of many eigenvalues in a large interval (we compared $K=6$ only). The advantages of the
scan lie elsewhere: (i) sub-windows are independent and can be computed in parallel (not exploited
in the timings above), whereas deflation is sequential; (ii) each detected eigenvalue carries its own
a posteriori enclosure (Weinstein, Kato--Temple), computed from the corresponding network alone and
independent of the accuracy of other eigenpairs; (iii) no orthogonalization against lower
eigenfunctions is needed, so errors do not accumulate with the index (for deflation the number of
lower eigenvalues grows like $E^{d/2}$ by Weyl's law); however, the scan has to reach a high window,
since cold starts are unreliable (Section~\ref{sec.mu}); (iv) the scan applies verbatim to nonlinear
problems such as the $p$-Laplacian, where no orthogonality is available. Its limitations are equally
clear: the cost of a full scan is (number of nodes) $\times$ (cost of one training), the second
derivatives make each epoch more expensive, and multiplicities and close clusters are not resolved
(Remark~\ref{rem.stability}).
\end{revblock}

\section{Conclusions}\label{sec.conclusions}

The method proposed in this work provides a flexible framework for computing eigenvalues and eigenfunctions of differential operators using neural networks. \rev{The detection step does not rely on the Rayleigh quotient or on orthogonality conditions between eigenfunctions; in the linear case the Rayleigh quotient of the detected network is then used to refine the value and to certify it.} As a result, it is well-suited for computing \rev{eigenvalues in a prescribed window} and for addressing nonlinear problems where such structures are no longer available.

\begin{revblock}
The main contributions of this work are: the characterization of eigenvalues as the zeros of a
residual loss curve, which needs neither the Rayleigh quotient nor orthogonality and therefore extends
to the $p$-Laplacian; the analysis of the penalized loss used in practice, with a rule for the penalty
parameter and a stability result for the detection; and a posteriori certificates for the computed
eigenvalues in the linear case, together with a systematic validation against exact values and
independent reference computations.
For self-adjoint problems the loss curve equals the squared distance to the spectrum; the penalized
loss used in practice has an explicit loss curve and no spurious local minima, which yields the rule
$\mu\ge D^2$ along warm-started scans, confirmed on a window with two eigenvalues. The detection is
stable under training and sampling errors whenever the measured error level is small compared with
the squared gaps, which holds at all nodes except collapsed or branch-switched ones, and these are
visible in the network; each eigenvalue comes with an a posteriori enclosure. The method was
validated against exact values in dimensions two to four and independent finite element computations in
dimensions two and three
(and tested in dimension ten, where the cost per epoch grows only linearly with the dimension: a
single run on a window chosen around the known eigenvalues, with a quotient depending on the node,
$+0.07\%$ to $+1.0\%$, $r=4.7$--$7$ and no Monte Carlo standard error computed),
with several seeds (seed means within $0.1\%$ and single runs within $0.25\%$ of the finite element
values for the potentials). For the $p$-Laplacian the zeros of the loss curve are exactly the
eigenvalues and the curve lies below a parabola at each of them; the first and second eigenvalues
agree with independent references to within $0.1\%$, further eigenvalues are detected in the same
scan and agree with independent references, and the dominant error sources (optimization noise and
the residual floor at critical points) have been identified; the method is validated for
$p\in[2.2,3]$, and for $p=1.5$ a regularized operator removes the failure of the strong form in a
single-seed test, with a bias of about $0.3\%$ of unidentified origin.

\emph{Limitations.} The scan does not resolve multiplicities or clusters of eigenvalues closer than
the error level; eigenvalues can be missed when the warm-started network stays on a neighbouring branch
(annulus, harmonic potential with $\omega=10$); training from scratch at a high value of $E$ is
unreliable, so high windows are reached by continuation or by repeated restarts; for the lowest
eigenvalues of linear problems a Rayleigh-quotient PINN with deflation is cheaper, and finite elements
are much cheaper and more accurate in dimensions two and three; no bound on the training error
is proved; and for the $p$-Laplacian there is no a posteriori certificate, the accuracy degrades for
$p\ge4$, and for $p<2$ the strong form requires a regularization of the operator, which changes the
eigenvalue slightly. A weak-form (variational) version
of the loss, which avoids the regularity required by the strong form, is a natural direction for
future work.
\end{revblock}

\appendix

\section{The minimization problem}

In this appendix, we show that the minimization problem leading to the loss curve $\Lambda(E)$ defined in \eqref{loss-curve} admits a solution. In order to keep things simple, we consider the case in which $V(x)$ is \rev{the hard-wall} potential \eqref{conf.V} with $\Omega\subset \R^{\rev{d}}$ bounded. \rev{(Given \eqref{cota.igual}, the existence also follows from the spectral decomposition, the minimizers being the eigenfunctions of the eigenvalue(s) closest to $E$; the proof below does not use the spectral theorem.)}

To this end, we must specify the function space over which the infimum is taken. The natural choice is the Sobolev space $H^2(\Omega) \cap H^1_0(\Omega)$. The main result of this section reads as follows:

\begin{theorem}\label{teo.min}
Let $V(x)$ be as in \eqref{conf.V} with $\Omega\subset \R^{\rev{d}}$ bounded \rev{and either convex or of class $C^{1,1}$}. Given \rev{$E\in\R$}, there exists $u_E\in H^2(\Omega)\cap H^1_0(\Omega)$ such that
$$
\Lambda(E) = \Lambda(u_E, E) = \inf_{u\in H^2(\Omega)\cap H^1_0(\Omega)\rev{\setminus\{0\}}}\Lambda(u, E),
$$
where $\Lambda(u, E)$ is given in \eqref{Lambda}.
\end{theorem}

\begin{proof}
The result follows from the \emph{direct method in the Calculus of Variations}. Fix \rev{$E\in\R$}, and consider a minimizing sequence $\{u_j\}_{j \in \N} \subset H^2(\Omega) \cap H^1_0(\Omega)$ such that $\|u_j\|_2 = 1$ and
$$
\Lambda(E) = \lim_{j \to \infty} \Lambda(u_j, E).
$$
In particular, the sequence $\Lambda(u_j, E)$ is bounded and nonnegative:
\begin{equation}\label{cota.Lambda}
0 \le \Lambda(u_j, E) \le C,
\end{equation}
for some constant $C > 0$. Now observe that
\begin{align*}
\Lambda(u_j, E) &= \int_\Omega (\Delta u_j + Eu_j)^2\, dx = \int_\Omega (\Delta u_j)^2\, dx + 2E\int_\Omega u_j \Delta u_j \, dx+ E^2.
\end{align*}
\rev{If $E=0$ there is nothing to do. Otherwise, using} the inequality $ab\le \theta a^2 + \frac{b^2}{4\theta}$ for any $a, b\rev{\ge}0$ and $\theta >0$ we estimate:
$$
\int_\Omega |u_j| |\Delta u_j|\, dx \le \theta \int_\Omega (\Delta u_j)^2\, dx + \frac{1}{4\theta}\int_\Omega u_j^2\, dx.
$$
Taking \rev{$\theta=1/(4|E|)$}, we obtain \rev{$2|E|\int_\Omega|u_j||\Delta u_j|\,dx\le\frac12\int_\Omega(\Delta u_j)^2\,dx+2E^2$, and therefore}
\begin{equation}\label{cota.Lambda2}
\Lambda(u_j, E)\ge \frac12 \int_\Omega (\Delta u_j)^2\, dx \rev{-E^2}.
\end{equation}
Combining \eqref{cota.Lambda}, \eqref{cota.Lambda2} \rev{we get $\|\Delta u_j\|_2^2\le2(C+E^2)$. Since $\Omega$ is convex or of class $C^{1,1}$, $\|u\|_{H^2(\Omega)}\le C_\Omega\|\Delta u\|_2$ for $u\in H^2(\Omega)\cap H^1_0(\Omega)$, so} the sequence $\{u_j\}_{j\in\N}$ is bounded in $H^2(\Omega)\cap H^1_0(\Omega)$.

Therefore, up to a subsequence, we have $u_j \rightharpoonup u_E$ weakly in $H^2(\Omega) \cap H^1_0(\Omega)$ for some $u_E$ in that space.

Since the functional
$$
u \mapsto \int_\Omega (\Delta u + E u)^2\, dx
$$
is convex in $u$, it is weakly lower semicontinuous. Thus,
$$
\int_\Omega (\Delta u_E + E u_E)^2\, dx \le \liminf_{j \to \infty} \int_\Omega (\Delta u_j + E u_j)^2\, dx.
$$
Moreover, the \rev{embedding} $H^2(\Omega)\cap H^1_0(\Omega)\hookrightarrow L^2(\Omega)$ is compact, so
$$
\int_\Omega u_E^2\, dx = \lim_{j\to\infty} \int_\Omega u_j^2\, dx = 1.
$$
Hence, $u_E$ is admissible and attains the minimum.
\end{proof}

\section{Proofs of Propositions~\ref{prop.penalized}, \ref{prop.rayleigh}, \ref{prop.plap}, \ref{prop.ck} and \ref{prop.stability}}\label{app.proofs}
\begin{revblock}
Throughout, $L$, $u_k$, $E_k$, $\ell(E)$ are as in Section~\ref{sec.approach} and $\|\cdot\|=\|\cdot\|_2$.

\begin{proof}[Proof of Proposition~\ref{prop.penalized}]
Write $u=\sum_ka_ku_k$, $s=\|u\|^2$. Then $\|(L-E)u\|^2=\sum_k(E_k-E)^2a_k^2\ge \ell^2s$, with equality
iff $u\in N_E$, the span of the eigenfunctions with $|E_k-E|=\ell$. Hence
$\bar\Lambda(u,E)\ge g(s):=\ell^2s+\mu(s-1)^2$, with equality iff $u\in N_E$. If $\ell^2<2\mu$, $g$ attains its
minimum on $[0,\infty)$ only at $s_*=1-\ell^2/(2\mu)$, with $g(s_*)=\ell^2-\ell^4/(4\mu)$; if $\ell^2\ge2\mu$, $g$ is
increasing and its minimum is $g(0)=\mu$. This gives the formula and the minimizers; the residual term
at the minimizer is $\ell^2s_*$. For the critical points, the derivative of $\bar\Lambda(\cdot,E)$ at $u$ in
the direction $u_k$ is $2[(E_k-E)^2-2\mu(1-s)]a_k$, so $u\ne0$ is critical iff all $k$ with
$a_k\ne0$ share the same $\lambda=(E_k-E)^2$ and $s=1-\lambda/(2\mu)$ (thus $\lambda<2\mu$). If
$\lambda>\ell^2$ and $u_j$ is an eigenfunction with $|E_j-E|=\ell$, then $a_j=0$ and
$\bar\Lambda(u+tu_j,E)=\bar\Lambda(u,E)+t^2(\ell^2-\lambda)+\mu t^4$, which is smaller for small $t\ne0$; so
$u$ is a saddle. Likewise $\bar\Lambda(tu_j,E)=\mu+t^2(\ell^2-2\mu)+\mu t^4$, so $u=0$ is not a local
minimizer if $\ell^2<2\mu$. The critical points with $\lambda=\ell^2$ are the global minimizers.
\end{proof}

\begin{proof}[Proof of Proposition~\ref{prop.rayleigh}]
Assume $\|u\|=1$, $u=\sum_ja_ju_j$, $\rho=\rho(u)$, $\eta=\eta(u)$. (a) follows from
$\|(L-E)u\|^2=\|(L-\rho)u\|^2+(\rho-E)^2$, since $\langle(L-\rho)u,u\rangle=0$. (b) $\|(L-E)u\|^2=
\sum_j(E_j-E)^2a_j^2\ge \ell(E)^2$. (c) Let $\beta=\rho+\gamma$. Since $\sigma\cap(E_k,\beta)=\emptyset$,
$(E_j-E_k)(E_j-\beta)\ge0$ for all $j$; expanding around $\rho$ and using $\sum a_j^2=1$,
$\sum a_j^2(E_j-\rho)=0$, $\sum a_j^2(E_j-\rho)^2=\eta^2$ gives $0\le\eta^2+(\rho-E_k)(\rho-\beta)$, i.e.\
$\rho-E_k\le\eta^2/\gamma$. The lower bound follows in the same way with $\alpha=\rho-\gamma$.
\end{proof}

\begin{proof}[Proof of Proposition~\ref{prop.stability}]
Write $\ell_i=\ell(E^i)$. (a) $\ell_i^2\le\tilde\Lambda_i+\delta\le\epsilon+\delta$. (b) Some node satisfies
$|E^i-E_k|\le h/2$, and then $\tilde\Lambda_i\le \ell_i^2+\delta\le h^2/4+\delta$. (c) Let $E_k$ be the
eigenvalue with $|E^m-E_k|=\ell_m$ and assume $\ell_m\ge h$. The neighbour $E^{m'}$ on the side of $E_k$ has
$\ell_{m'}\le \ell_m-h$, so $\ell_m^2-\delta\le\tilde\Lambda_m\le\tilde\Lambda_{m'}\le(\ell_m-h)^2+\delta$, i.e.\
$2h\ell_m\le h^2+2\delta$. (d) Let $I=[E_k-\tau-h,E_k+\tau+h]$ and let $E^m$ minimize $\tilde\Lambda$ over
the nodes in $I$, and $\gamma_k=\dist(E_k,\sigma\setminus\{E_k\})$. By (b), $\tilde\Lambda_m\le h^2/4+\delta$, hence $\ell_m\le\tau$. If the eigenvalue closest
to $E^m$ were $E_j\ne E_k$, then $|E^m-E_k|\ge\gamma_k-\tau>\tau+h$, contradicting $E^m\in I$; hence
$|E^m-E_k|\le\tau$, both neighbours of $E^m$ lie in $I$, and $E^m$ is a discrete local minimum.
\end{proof}

\begin{proof}[Proof of Proposition~\ref{prop.plap}]
(a) Since $\Delta_pu_k=-E_k|u_k|^{p-2}u_k$, the residual is $(E-E_k)|u_k|^{p-2}u_k$ and the penalty
vanishes. (b) If $E\in\sigma_p$, $\Lambda_p(E)=0$ by (a). Conversely, by the homogeneity
$\Delta_p(tu)=|t|^{p-2}t\Delta_pu$, every $u\ne0$ can be written as $tw$ with $\|w\|_p=1$, and
$\bar\Lambda_p(tw,E)=t^{2(p-1)}Q(w)+\mu(t^p-1)^2$ with $Q(w)=\int(\Delta_pw+E|w|^{p-2}w)^2$; hence
$\Lambda_p(E)=0$ forces $\inf_{\|w\|_p=1}Q(w)=0$. Take $w_n$ with $\|w_n\|_p=1$ and
$f_n:=-\Delta_pw_n-E|w_n|^{p-2}w_n\to0$ in $L^2$. Testing with $w_n$ and using
$\|w\|_2\le C_S\|\nabla w\|_p$ (here $p\ge2d/(d+2)$ is used),
$\|\nabla w_n\|_p^p\le E+C_S\|f_n\|_2\|\nabla w_n\|_p$, so $w_n$ is bounded in $W^{1,p}_0$. Up to a
subsequence $w_n\rightharpoonup w$ in $W^{1,p}_0$ and $w_n\to w$ in $L^p$, so $\|w\|_p=1$ and
$|w_n|^{p-2}w_n\to|w|^{p-2}w$ in $L^{p'}\subset W^{-1,p'}$; also $f_n\to0$ in $W^{-1,p'}$. Thus
$-\Delta_pw_n\to E|w|^{p-2}w$ in $W^{-1,p'}$, and by continuity of $(-\Delta_p)^{-1}\colon
W^{-1,p'}\to W^{1,p}_0$, $w_n\to w$ strongly, with $-\Delta_pw=E|w|^{p-2}w$, $w\ne0$.
\end{proof}

\begin{proof}[Proof of Proposition~\ref{prop.ck}]
Write $u=u_k$ and $q=2(p-1)$. Testing the weak form of \eqref{p-eigen} with $u$ gives
$\|\nabla u\|_p^p=E_k\|u\|_p^p=E_k$. (b) is $\|u\|_2=1$.

(a) Let $1<p<2$, so that $0<q<p$. H\"older's inequality with exponents $p/q$ and $p/(p-q)$ gives
$c_k=\int_\Omega|u|^q\le|\Omega|^{1-q/p}\big(\int_\Omega|u|^p\big)^{q/p}=|\Omega|^{(2-p)/p}$, since
$1-q/p=(2-p)/p$. For the lower bound note that $p<2\le d$, so $p^*$ is defined and $q<p<p^*$. Writing
$|u|^p=|u|^{p\theta}|u|^{p(1-\theta)}$ and applying H\"older's inequality with exponents
$q/(p\theta)$ and $p^*/(p(1-\theta))$, whose reciprocals add up to $1$ by the definition of $\theta$, we
obtain the interpolation inequality
\[
1=\|u\|_p\le\|u\|_q^{\theta}\|u\|_{p^*}^{1-\theta}\le c_k^{\theta/q}\big(S_*\|\nabla u\|_p\big)^{1-\theta}
=c_k^{\theta/q}\big(S_*E_k^{1/p}\big)^{1-\theta}.
\]
Raising to the power $q/\theta$ gives $c_k\ge S_*^{-q(1-\theta)/\theta}E_k^{-q(1-\theta)/(p\theta)}$.

(c) Let $p>2$, so that $q>p$. H\"older's inequality with exponents $q/p$ and $q/(q-p)$ gives
$1=\int_\Omega|u|^p\le|\Omega|^{1-p/q}c_k^{p/q}$, i.e.\ $c_k\ge|\Omega|^{-(q-p)/p}=|\Omega|^{(2-p)/p}$, since
$q-p=p-2$. For the upper bound we need $W^{1,p}_0(\Omega)\hookrightarrow L^q(\Omega)$. If $p>d$ this follows
from Morrey's embedding into $L^\infty(\Omega)$ and the boundedness of $\Omega$; if $p=d$,
$W^{1,d}_0(\Omega)$ embeds into $L^s(\Omega)$ for every $s<\infty$; if $p<d$, the Sobolev embedding into
$L^{p^*}(\Omega)$ and H\"older's inequality give the embedding when $q\le p^*$, and
$q\le p^*\iff2(p-1)(d-p)\le dp\iff d(p-2)\le2p(p-1)$. Then
$c_k=\|u\|_q^q\le S_q^q\|\nabla u\|_p^q=S_q^qE_k^{q/p}$. Finally, writing $p=2+t$ with $t>0$,
$\frac{2p(p-1)}{p-2}=2\big(t+3+\frac2t\big)\ge6+4\sqrt2>11$, so the condition holds for all $p>2$ when
$d\le11$.
\end{proof}
\end{revblock}

\section{Convergence study for the \texorpdfstring{$p$}{p}-Laplacian}\label{app.conv}
\begin{revblock}
Tables~\ref{tab.plap.conv} and~\ref{tab.plap.conv2} collect the convergence studies for the first and second eigenvalues of the $p$-Laplacian discussed in Section~\ref{sec.nolin}.
\begin{table}[H]
\centering\small

\caption{Convergence study for the first eigenvalue of the $p$-Laplacian, $p=2.2$, in the unit disk ($E_1=6.53174$). Each row changes one setting of the baseline ($3\times128$, $N=5000$, 1500 epochs with learning-rate decay, $\mu=200$, $h=0.5$); the loss is sampled at $E=6.0,6.5,7.0$ on $4\cdot10^4$ fresh points. P0: constant learning rate $10^{-3}$ with early stopping. Single runs (seed 0) except where the seed is given.}
\label{tab.plap.conv}
\begin{tabular}{lccccc}
\toprule
setting & $\bar\Lambda_p(6.0)$ & $\bar\Lambda_p(6.5)$ & $\bar\Lambda_p(7.0)$ & $E_{\rm curve}$ (rel.\ diff.\ \%) & $E_{\rm PINN}$ (rel.\ diff.\ \%)\\
\midrule
baseline, seed 0 & 0.394 & 0.131 & 0.316 & 6.543 ($+0.18$) & 6.5320 ($+0.00$)\\
baseline, seed 1 & 0.416 & 0.130 & 0.329 & 6.545 ($+0.20$) & 6.5329 ($+0.02$)\\
baseline, seed 2 & 0.401 & 0.125 & 0.318 & 6.544 ($+0.19$) & 6.5285 ($-0.05$)\\
P0, seed 0 & 0.439 & 0.348 & 0.426 & 6.520 ($-0.17$) & 6.5393 ($+0.12$)\\
P0, seed 1 & 0.427 & 0.204 & 0.338 & 6.562 ($+0.47$) & 6.5453 ($+0.21$)\\
P0, seed 2 & 0.419 & 0.272 & 0.296 & 6.679 ($+2.26$) & 6.5557 ($+0.37$)\\
network $2\times32$ & 0.534 & 0.150 & 0.353 & 6.577 ($+0.69$) & 6.5350 ($+0.05$)\\
network $2\times64$ & 0.467 & 0.137 & 0.348 & 6.555 ($+0.36$) & 6.5343 ($+0.04$)\\
network $4\times256$ & 0.436 & 0.096 & 0.307 & 6.559 ($+0.41$) & 6.5338 ($+0.03$)\\
$N=2000$ & 0.417 & 0.127 & 0.332 & 6.543 ($+0.17$) & 6.5403 ($+0.13$)\\
$N=20000$ & 0.397 & 0.123 & 0.325 & 6.538 ($+0.10$) & 6.5319 ($+0.00$)\\
500 epochs & 0.515 & 0.153 & 0.338 & 6.581 ($+0.76$) & 6.5415 ($+0.15$)\\
5000 epochs & 0.376 & 0.049 & 0.244 & 6.563 ($+0.48$) & 6.5315 ($-0.00$)\\
$\mu=50$ & 0.390 & 0.073 & 0.250 & 6.571 ($+0.60$) & 6.5317 ($-0.00$)\\
$\mu=1000$ & 0.516 & 0.166 & 0.385 & 6.558 ($+0.39$) & 6.5352 ($+0.05$)\\
$h=0.25$ & -- & 0.120 & -- & 6.583 ($+0.78$) & 6.5338 ($+0.03$)\\
$h=0.125$ & -- & 0.116 & -- & -- (monotone) & 6.5337 ($+0.03$)\\
\bottomrule
\end{tabular}
\end{table}

\begin{table}[H]
\centering\small

\caption{Convergence study for the second eigenvalue of the $p$-Laplacian, $p=2.2$, in the unit disk (reference $18.3929$: finite elements on the half disk; $18.395$ in \cite{HJ11}). Nodes $E=18.0,18.5,19.0$; $E_{\rm PINN}$ from the network trained at $E=18.5$; single runs (seed 0) except the three baseline seeds.}
\label{tab.plap.conv2}
\begin{tabular}{lccccc}
\toprule
setting & $\bar\Lambda_p(18.0)$ & $\bar\Lambda_p(18.5)$ & $\bar\Lambda_p(19.0)$ & $E_{\rm curve}$ (rel.\ diff.\ \%) & $E_{\rm PINN}$ (rel.\ diff.\ \%)\\
\midrule
baseline, seed 0 & 1.056 & 0.832 & 1.292 & 18.41 ($+0.11$) & 18.4069 ($+0.08$)\\
baseline, seed 1 & 1.066 & 0.890 & 1.160 & 18.45 ($+0.30$) & 18.3890 ($-0.02$)\\
baseline, seed 2 & 1.144 & 0.876 & 1.259 & 18.46 ($+0.34$) & 18.3968 ($+0.02$)\\
P0 (constant learning rate, early stopping) & 1.103 & 0.937 & 1.226 & 18.43 ($+0.22$) & 18.3867 ($-0.03$)\\
network $2\times64$ & 1.576 & 1.123 & 1.284 & 18.62 ($+1.23$) & 18.3823 ($-0.06$)\\
5000 epochs & 1.087 & 0.712 & 0.753 & 18.70 ($+1.68$) & 18.3857 ($-0.04$)\\
\bottomrule
\end{tabular}
\end{table}
\end{revblock}

\section*{Declaration}
\begin{revblock}
\noindent\textbf{Conflicts of interest:} The authors declare that they have no known competing financial
interests or personal relationships that could have appeared to influence the work reported in this paper.

\noindent\textbf{Funding:} This work was partially supported by grants UBACYT Prog.\ 2018
20020170100445BA and ANPCyT PICT 2019-00985.

\noindent\textbf{Author's Contribution:} J.~Fern\'andez Bonder and A.~Salort contributed equally to all
aspects of this work (conceptualization, methodology, formal analysis, software, investigation,
writing -- original draft, writing -- review and editing). All authors reviewed the results and
approved the final version of the manuscript.

\noindent\textbf{Acknowledgement:} JFB thanks Pablo Groisman for encouraging him to teach a course on
the Mathematical Foundations of Deep Learning in the fall semester of 2025, which inspired this work.

\noindent\textbf{Data availability statement:} No external data were used. The Python code, the scripts
of all the experiments reported here and their raw results are openly accessible at
\url{https://github.com/jfbonder/PINN-eigenvalues-v2}.

\noindent\textbf{Using of AI tools:} During the preparation of this revised version, the authors used
Claude (Anthropic) through a multi-agent system (models Claude Fable 5.2, Claude Opus 5.5 and Claude
Sonnet 5.5). The code was originally developed by the authors and was subsequently modified with the
assistance of these AI tools, which were also used to implement and run additional numerical
experiments (including the finite element reference computations) and to draft and edit parts of the
text. The authors checked all results and the final text and take full responsibility for the content
of this article.
\end{revblock}


\begin{thebibliography}{99}

\bibitem{Lu-Meng-Mao}
Lu L., Meng X., Mao Z., Karniadakis G.E., DeepXDE: a deep learning library for solving differential
equations, SIAM Rev., 63, 208, 2021.

\bibitem{Rackauckas}
Rackauckas C., Ma Y., Martensen J., Warner C., Zubov K., Supekar R., Skinner D., Ramadhan A.,
Edelman A., Universal differential equations for scientific machine learning, arXiv preprint,
\href{https://arxiv.org/abs/2001.04385}{arXiv:2001.04385}, 2021.

\bibitem{Rackauckas-Innes-Ma-Bettencourt-etal}
Rackauckas C., Innes M., Ma Y., Bettencourt J., White L., Dixit V., DiffEqFlux.jl -- a Julia library
for neural differential equations, arXiv preprint,
\href{https://arxiv.org/abs/1902.02376}{arXiv:1902.02376}, 2019.

\bibitem{Blechschmidt-Ernst}
Blechschmidt J., Ernst O.G., Three ways to solve partial differential equations with neural
networks---a review, GAMM-Mitt., 44, e202100006, 2021.

\bibitem{Berg-Nystrom}
Berg J., Nystr\"om K., A unified deep artificial neural network approach to partial differential
equations in complex geometries, Neurocomputing, 317, 28, 2018.

\bibitem{GT24}
Ghazi F., Tawfiq L., Design optimal neural network based on new LM training algorithm for solving
3D-PDEs, An International Journal of Optimization and Control: Theories \& Applications, 14, 249, 2024.

\bibitem{KKG}
Kumar A., Kumar M., Goswami P., Numerical solution of coupled system of Emden-Fowler equations using
artificial neural network technique, An International Journal of Optimization and Control: Theories \&
Applications, 14, 2024.

\bibitem{MHO}
Madani N., Hammouch Z., Ozdemir N., Hybrid GA-IPA and neural network-based solution for
fractional-order nonlinear predator-prey dynamical systems, Nonlinear Dyn., 113, 32979, 2025.

\bibitem{Yasin-etal}
Yasin F., Alshehri M.H., Arshad M., Shang Y., Afzal Z., Exploring dynamics of multi-peak and
breathers-type solitary wave solutions in generalized higher-order nonlinear Schr\"odinger equation
and their optical applications, Alex. Eng. J., 105, 402--413, 2024.

\bibitem{Lagaris-Likas-Fotiadis}
Lagaris I.E., Likas A., Fotiadis D.I., Artificial neural networks for solving ordinary and partial
differential equations, IEEE Trans. Neural Netw., 9, 987, 1998.

\bibitem{Raissi-Perdiakaris-Karniadakis}
Raissi M., Perdikaris P., Karniadakis G.E., Physics-informed neural networks: A deep learning
framework for solving forward and inverse problems involving nonlinear partial differential
equations, J. Comput. Phys., 378, 686, 2019.

\bibitem{Karniadakis-etal}
Karniadakis G.E., Kevrekidis I.G., Lu L., Perdikaris P., Wang S., Yang L., Physics-informed machine
learning, Nat. Rev. Phys., 3, 422, 2021.

\bibitem{Jin-Mattheakis-Protopapas1}
Jin H., Mattheakis M., Protopapas P., Unsupervised neural networks for quantum eigenvalue problems,
Proceedings of the 2020 NeurIPS Workshop on Machine Learning and the Physical Sciences, Vancouver,
Canada, 2020.

\bibitem{Jin-Mattheakis-Protopapas2}
Jin H., Mattheakis M., Protopapas P., Physics-informed neural networks for quantum eigenvalue
problems, Proceedings of the 2022 International Joint Conference on Neural Networks (IJCNN), IEEE,
2022.

\bibitem{Sakar}
Sarkar S., Physics-informed neural networks for one-dimensional quantum well problems, arXiv
preprint, \href{https://arxiv.org/abs/2504.05367}{arXiv:2504.05367}, 2025.

\bibitem{Holliday-Lindner-Ditto}
Holliday E.G., Lindner J.F., Ditto W.L., Solving quantum billiard eigenvalue problems with
physics-informed machine learning, AIP Adv., 13, 085013, 2023.

\bibitem{Rowan-Evans-Maute-Dootsan}
Rowan C., Evans J., Maute K., Doostan A., Solving engineering eigenvalue problems with neural
networks using the Rayleigh quotient, Int. J. Numer. Methods Eng., 126(24), e70209, 2025.

\bibitem{E-Yu}
E W., Yu B., The Deep Ritz method: a deep learning-based numerical algorithm for solving variational
problems, Commun. Math. Stat., 6, 1, 2018.

\bibitem{Betcke-Trefethen}
Betcke T., Trefethen L.N., Reviving the method of particular solutions, SIAM Rev., 47, 469, 2005.

\bibitem{Wang-Zunger}
Wang L.-W., Zunger A., Solving Schr\"odinger's equation around a desired energy: application to
silicon quantum dots, J. Chem. Phys., 100, 2394, 1994.

\bibitem{Weinstein1934}
Weinstein D.H., Modified Ritz method, Proc. Natl. Acad. Sci. USA, 20, 529, 1934.

\bibitem{Parlett1998}
Parlett B.N., The Symmetric Eigenvalue Problem, Classics in Applied Mathematics 20, SIAM,
Philadelphia, USA, 1998.

\bibitem{Kato1949}
Kato T., On the upper and lower bounds of eigenvalues, J. Phys. Soc. Japan, 4, 334, 1949.

\bibitem{Anane1987}
Anane A., Simplicit\'e et isolation de la premi\`ere valeur propre du $p$-laplacien avec poids,
C. R. Acad. Sci. Paris S\'er. I Math., 305, 725, 1987.

\bibitem{Lindqvist1990}
Lindqvist P., On the equation $\operatorname{div}(|\nabla u|^{p-2}\nabla u)+\lambda|u|^{p-2}u=0$,
Proc. Amer. Math. Soc., 109, 157, 1990.

\bibitem{GarciaAzorero-Peral}
Garc\'ia Azorero J.P., Peral Alonso I., Existence and nonuniqueness for the $p$-Laplacian: nonlinear
eigenvalues, Comm. Partial Differential Equations, 12, 1389, 1987.

\bibitem{Anane-Tsouli}
Anane A., Tsouli N., On the second eigenvalue of the $p$-Laplacian, Nonlinear Partial Differential
Equations, Pitman Res. Notes Math. Ser., 343, 1--9, 1996.

\bibitem{Cuesta-deFigueiredo-Gossez}
Cuesta M., de Figueiredo D., Gossez J.-P., The beginning of the Fu\v{c}ik spectrum for the
$p$-Laplacian, J. Differential Equations, 159, 212, 1999.

\bibitem{HJ11}
Hor\'ak J., Numerical investigation of the smallest eigenvalues of the $p$-Laplace operator on planar
domains, Electron. J. Differential Equations, 2011(132), 1--30, 2011.

\bibitem{DiBenedetto}
DiBenedetto E., $C^{1+\alpha}$ local regularity of weak solutions of degenerate elliptic equations,
Nonlinear Anal., 7, 827, 1983.

\bibitem{Tolksdorf}
Tolksdorf P., Regularity for a more general class of quasilinear elliptic equations, J. Differential
Equations, 51, 126, 1984.

\bibitem{Hornik1991}
Hornik K., Approximation capabilities of multilayer feedforward networks, Neural Netw., 4, 251, 1991.

\bibitem{Pinkus1999}
Pinkus A., Approximation theory of the MLP model in neural networks, Acta Numer., 8, 143, 1999.

\bibitem{DeRyck-Lanthaler-Mishra}
De Ryck T., Lanthaler S., Mishra S., On the approximation of functions by tanh neural networks,
Neural Netw., 143, 732, 2021.

\bibitem{Rahaman-etal}
Rahaman N., Baratin A., Arpit D., Draxler F., Lin M., Hamprecht F., Bengio Y., Courville A., On the
spectral bias of neural networks, Proceedings of the 36th International Conference on Machine
Learning, Proc. Mach. Learn. Res., 97, 5301, 2019.

\end{thebibliography}
\end{document}